\documentclass[11pt,reqno]{amsart}
\usepackage{geometry}                
\usepackage{graphicx}
\usepackage{amssymb}
\usepackage{epstopdf}
\usepackage{amsmath}
\usepackage{bm}
\usepackage{color}
\usepackage{hyperref}
\usepackage{pdfsync}
\usepackage{float}
\usepackage{verbatim}
\usepackage{tikz}
\usetikzlibrary{arrows,decorations.pathmorphing,backgrounds,positioning,fit,petri}
\usetikzlibrary{cd,patterns,matrix}
\usetikzlibrary{decorations.markings}
\usetikzlibrary {arrows.meta}
\usetikzlibrary{bending}

\usepackage[all,cmtip]{xy}
\xyoption{matrix}
\xyoption{arrow}
\usepackage[table]{colortbl}

\tikzset{help lines/.style={step=#1cm,very thin, color=gray},
help lines/.default=.5} 
\tikzset{thick grid/.style={step=#1cm,thick, color=gray},
thick grid/.default=1} 

\usepackage{tcolorbox}  
\tcbset{colframe=black,colback=white, width=(\linewidth-2mm), before=,after=\hfill,}
\usepackage{graphicx}
\usepackage{tikz}
\usetikzlibrary{mindmap,shapes,snakes,matrix}

\tikzstyle{ann}=[fill=white, inner sep=1pt, font=\footnotesize{#1}]
\tikzstyle{annfar}=[inner sep=2pt, font=\footnotesize{#1}]
\tikzstyle{annfarer}=[inner sep=3pt, font=\footnotesize{#1}]
\tikzstyle{annrot}=[fill=white, text=blue!75!black, inner sep=1pt, font=\footnotesize{#1}]
\tikzstyle{wall}=[thick]
\tikzstyle{nullwall}=[thick, dotted]

\newtheorem{thm}{Theorem}[section]
\newtheorem{lem}[thm]{Lemma}
\newtheorem{cor}[thm]{Corollary}
\newtheorem{prop}[thm]{Proposition}
\newtheorem{defn}[thm]{Definition}
\newtheorem{exm}[thm]{Example}

 \newtheorem{prop-def}[thm]{Proposition-Definition}
\newtheorem{lem-def}[thm]{Lemma-Definition}
\newtheorem{rmk}[thm]{Remark}

\numberwithin{equation}{section}

\DeclareMathOperator{\Hom}{Hom}%
\DeclareMathOperator{\Ext}{Ext}%
\DeclareMathOperator{\add}{add} 
\DeclareMathOperator{\pd}{projdim}%
\DeclareMathOperator{\Gen}{Gen}%

\DeclareMathOperator{\Homeo}{Homeo}%

\newcommand{\commentout}[1]{}

\newcommand{\modu}{\mbox{-}\operatorname{mod}}
\newcommand{\pwf}{\mbox{-}\operatorname{pwf}}
\newcommand{\Int}{\mbox{-}\operatorname{int}}
\newcommand{\tilt}{\mbox{-}\operatorname{tilt}}
\newcommand{\Intv}{\operatorname{Intv}}
\newcommand{\MRig}{\operatorname{MR}}
\newcommand{\MCom}{\operatorname{MC}}
\newcommand{\Rig}{\operatorname{Rig}}
\newcommand{\Com}{\operatorname{Com}}
\newcommand{\supp}{\operatorname{supp}}
\newcommand{\Supp}{\operatorname{Supp}}
\newcommand{\MRIC}{\operatorname{MRIC}}
\newcommand{\MCIC}{\operatorname{MCIC}}
\newcommand{\MRCA}{\operatorname{MRCA}}
\newcommand{\MRCB}{\operatorname{MRCB}}

 \newcommand{\bR}{{\mathbb {R}}}

   \newcommand{\bA}{{\mathbb {A}}}
 \newcommand{\bP}{\mathbb{P}}

\newcommand{\bZ}{\mathbb{Z}}
\newcommand{\fU}{\mathbf{U}}
\newcommand{\bI}{{\mathbb{I}}}
\newcommand{\bJ}{{\mathbb{J}}}
\newcommand{\fD}{\mathbf{D}}

\newcommand{\fT}{\mathbf{T}}
\newcommand{\suc}{{\rm{Suc}}}
\newcommand{\pre}{{\rm{Pre}}}

\newcommand{\cA}{\ensuremath{{\mathcal{A}}}}

\newcommand{\cF}{\ensuremath{{\mathcal{F}}}}

\newcommand{\cM}{\ensuremath{{\mathcal{M}}}}
\newcommand{\cN}{\ensuremath{{\mathcal{N}}}}

\newcommand{\cR}{\ensuremath{{\mathcal{R}}}}

\newcommand{\cT}{\ensuremath{{\mathcal{T}}}}

\newcommand{\rep}{\operatorname{rep}}
\newcommand{\Rep}{\operatorname{Rep}}

\newcommand{\cok}{\operatorname{Coker}}
\newcommand{\Ima}{\operatorname{Im}}
\newcommand{\bk}{{k}}

\newcommand{\mapsfrom}{\mathrel{\reflectbox{\ensuremath{\mapsto}}}}

\begin{document}

\title[Tilting theory on continuous quivers (I)]{Tilting theory on continuous quivers (I) \\ maximal rigid and tilting subcategories}
\author [Haowei Liu, Zonghui Liu, Shijie Zhu] {Haowei Liu, Zonghui Liu, Shijie Zhu$^*$}
\address{School of Mathematics and Statistics, Nantong University, Jiangsu 226019, P. R. China}
\thanks{The corresponding author S.Z. is sponsored by NSF 12201321.}
\thanks{3256700577@qq.com  \ \ 2352674950@qq.com\ \ shijiezhu$\symbol{64}$ntu.edu.cn}

 \maketitle

\begin{abstract}  Let $\mathrm{rep}(\mathbb I)$ be the category of finitely generated persistence modules on the interval $\mathbb I\subseteq\mathbb R$. We study the maximal rigid subcategories of $\mathrm{rep}(\mathbb I)$. Maximal rigid interval consistent subcategories are classified via tilting modules over the corresponding discrete type $A$ quiver. We compare several related notions, (weak) tilting and maximal rigidity, in both the context of hereditary abelian categories and $\mathrm{rep}(\mathbb I)$ and prove that every maximal rigid interval consistent subcategory is weak tilting and every maximal rigid consistent below subcategory is tilting. We also establish a continuous analogue of Ingalls-Thomas's bijection by showing a bijection between integrally supported maximal rigid consistent below subcategories and integrally generated torsion classes. 
  \vskip5pt 

{\it Key words and phrases. persistence module, interval consistency, maximal rigidity, (weak) tilting subcategories, integrally supported subcategories, integrally generated torsion classes.}

{\it MSC2020: 16G20, 16D90, 16S90.}
\end{abstract}

\setcounter{tocdepth}{1} 
\tableofcontents

 \section {Introduction}
 \subsection{Background}
 
Tilting theory origins from  Bernstein-Gelfand-Ponomarev's  reflection functors 
and the subsequent study of Coxeter functors by Auslander-Platzeck-Reiten~\cite{APR}. Brenner-Butler~\cite{BB} introduced the notion of tilting
modules, and Happel-Ringel~\cite{HR} developed the classical formulation that underlies much of the subsequent theory. Let $A$ be a finite dimensional algebra. A basic left $A$-module $T$ is called {\bf tilting} if 
 \begin{enumerate}
 \item[(T1)] $\pd T\leq 1$;
 \item[(T2)] $\Ext^1(T,T)=0$;
 \item[(T3)] There is an exact sequence $0\to A\to T_0\to T_1\to 0$, where $T_0$, $T_1\in\add T$.
 \end{enumerate}

As its wide applications, tilting theory has soon become one of the most important topics in representation theory. In the process of development,  some variations of tilting modules are studied from different perspectives.  Ingalls-Thomas \cite{IT2} introduced the notion of support tilting modules over hereditary algebras, which extends the classical theory by allowing a module to be tilting over its support algebra, and they established the following bijection: 
\begin{thm}\cite[Theorem 2.11]{IT2}\label{thm:IT}
	Let $Q$ be a finite acyclic quiver. Then there is a bijection between support tilting modules over $kQ$ and finitely generated torsion classes in $kQ$-mod.
\end{thm}
Adachi-Iyama-Reiten~\cite{AIR}  later generalized the notion to support $\tau$-tilting modules for arbitrary finite-dimensional algebras and showed a bijection between  support $\tau$-tilting modules and functorially finite torsion classes.  

There is another generalization in a different flavor, which replaces a single tilting module by an additive subcategory of an abelian category.
\begin{defn}\cite[Definition 2.17]{BBOS}\label{def:tilt_sub} A subcategory  $\fT$ of an abelian category $\cA$ is {\bf weak tilting} provided that
 \begin{enumerate}
 \item[(W1)] $\pd T\leq 1$ for each $T\in \fT$;
 \item[(W2)] $\Ext^1(T,T')=0$, for any $T,T'\in \fT$;
 \item[(W3)] For each projective object $P$ in $\cA$, there is an exact sequence $0\to P\to T_0\to T_1\to 0$, where $T_0$, $T_1\in\fT$.
 \end{enumerate}
 A weak tilting subcategory is {\bf tilting} if it is contravariantly finite. 
 \end{defn}
Notice that in the category of finitely generated modules over finite dimensional algebras, weak tilting and tilting subcategories coincide.
\begin{prop}\label{prop:tilting_eqv}
 Let $A$ be a finite dimensional algebra, then the following are equivalent.
 \begin{enumerate}
 \item $T$ is a tilting module;
 \item $\add T$ is a tilting subcategory;
 \item $\add T$ is a weak tilting subcategory;
 \item $\Gen(T)=T^{\perp_1} (:=\{X\in A\modu | \Ext^1(T,X)=0\})$;
 \end{enumerate} 
 Furthermore, if $A$ is hereditary, then these conditions are equivalent to 
 \begin{enumerate}
 \item[(5)] $T$ is a maximal rigid module.
 \end{enumerate}
 \end{prop}
 \begin{proof}
 The implication $(1)\iff(2)$ is straightforward. Since $\add T$ is representation finite-type, it is always contravariantly finite. Thus, $(2)\iff(3)$ holds. The implication $(1)\iff(4)$ is well-known \cite[VI Theorem 2.5]{ASS}. The implication $(1)\iff(5)$ follows from Bongartz completion Theorem \cite{Bon}.
 \end{proof}
 In this paper, we will study (weak) tilting subcategories in the context of representations of continuous quivers. Unlike tilting modules of path algebras, similar concepts such as maximal rigidity, weak tilting, and tilting may not coincide. Understanding their relationship is therefore essential when passing from module categories of finite dimensional algebras to categories of continuous representations.
 
\vskip5pt
Persistence modules arise in topological data analysis by applying homology functor to a filtration of spaces. Algebraically, a persistence module over a poset $\mathbb P$ is a functor from $\mathbb P$, viewed as a category, to the category of $k$-vector spaces. The  poset $\mathbb R$ is usually considered as a continuous counterpart of  linearly oriented type $A$ quiver, due to the similarity of their representations. A decomposition theorem of Crawley-Boevey~\cite{CB} provides a fundamental structural result: pointwise finite-dimensional persistence modules over $\mathbb R$ decompose uniquely, up to isomorphism, into interval modules. A finitely generated persistence $\bR$-module is finite direct sum of interval modules. Denote by $\rep(\bR)$ the full subcategory of finitely generated persistence $\bR$-modules.

Despite the motivation from persistence theory, studying continuous representations over posets has its own interest in representation theory. This can be dated back to Chaptal and Loupias's work  \cite{Ch, Lou} on representation types of finite posets. Igusa and Todorov  \cite{IT1} introduced the notion of continuous cluster categories, extending the classical cluster theory.  In a series work of Igusa, Rock and Todorov~\cite{IRT,R,IRT3,IRT4}, the authors developed the representation theory of continuous quivers of type $A$. A continuous analogue of Nakayama algebras was studied in \cite{RZ}. Thread quivers, as introduced by Berg and Van Roosmalen \cite{BV} and developed recently by Paquette, Rock and Yildirim \cite{PRY}, replace arrows between vertices of quivers by totally ordered sets, revealing the interplay between representations of discrete and continuous quivers. These approaches motivate the study of which features in the representation theory of finite-dimensional algebras still persist and which concepts can be generalized in representations of continuous quivers.

 \subsection{Main results}
Fix a field $k$ and integers $m,n$. We consider the four interval types 
\[
    \bI_{m,n}=[m,n],\qquad
    \bI^+_{m,n}=(m,n],\qquad
    \bI^-_{m,n}=[m,n),\qquad
    \bI^\pm_{m,n}=(m,n).
\] We use the notation $\bI^*_{m,n}$ when the interval type need not be specified. More general intervals can be treated via categorical equivalences induced by order-preserving homeomorphisms using these normalized models (see Proposition \ref{prop:homeo}).

A nonempty subinterval of $\bI^*_{m,n}$ is called  {\bf integral} if both endpoint values of the interval are integers, and {\bf half-integral} if at least one endpoint value is an integer. A representation is integral or half-integral if it is a finite direct sum of interval modules of
the corresponding kind. We denote the full subcategory of finitely generated persistence modules and integral representations by $\rep(\bI^*_{m,n})$ and $\bI^*_{m,n}\Int$ respectively. There is an equivalence of abelian categories $\fD^*_{m,n}:kQ^*_{m,n}\modu
\xrightarrow{\ \sim\ }\bI^*_{m,n}\Int$ (see Lemma \ref{lem:eqv_A}), which builds a bridge between $\bI^*_{m,n}$-representations and quiver representations.

In this paper, we study maximal rigid subcategories of the category of finitely generated representations over an interval $\bI\subseteq\bR$.  Recall that a subcategory $\fT$ of an abelian category is {\bf rigid} if $\Ext^1(T,T')=0$ for all $T,T'\in\fT$, and is {\bf maximal rigid} if it is maximal among rigid subcategories.  Liu-Zhao \cite{LZ1} studied maximal rigid representations on the closed interval $[0,1]$ subject to a condition determined by a finite sequence of critical points, and showed a connection between these subcategories and tilting modules \cite[Theorem 2.5]{LZ1}. 

We reformulate this crucial condition as the {\bf interval consistency} condition (see Definition \ref{def:int_consist}). In brief, an interval consistent set $\cM$ contains half-integral subintervals of $\bI_{m,n}^*$ such that for each $x\in(i,i+1)\subseteq\bI_{m,n}^*$,  $\cM$ contains exactly one open and one closed interval with one endpoint $x$ and the opposite endpoint remains fixed. If all the left (resp. right) endpoints of intervals in $\cM$ are integers, then $\cM$ is called {\bf consistent below (resp. above)}. An additive subcategory $\fT$ is called interval consistent (resp. consistent below (above)) if it consists of direct sums of interval modules $k_I$ where $I$ comes from an interval consistent (resp. consistent below (above)) set $\cM$. These conditions make continuous families of interval modules to be controlled by finitely many integral representations.

\vskip5pt
Our first goal is to  classify maximal rigid interval consistent subcategories over all four types of intervals $\bI_{m,n}^*$.
 
  
\begingroup
\setcounter{thm}{0}
\renewcommand\thethm{\Alph{thm}}

\begin{thm}[Corollary \ref{cor:rigid_complete}, Lemma \ref{lem:tilt}]\label{thm:A} The following statements hold.
\begin{enumerate} 
\item If $\fT$ is a maximal rigid interval consistent subcategory of $\rep(\bI_{m,n}^*)$, then $$\widehat\fT=\fT\cap\bI_{m,n}^*\Int$$ is a  maximal rigid subcategory of $\bI_{m,n}^*\Int$. 
\item The assignment (given by Lemma \ref{lem:eqv_A})
   $$
        [T]\longmapsto\operatorname{add}\fD^*_{m,n}(T)
   $$
   is a bijection between isomorphism classes of basic tilting $kQ^*_{m,n}$-modules and maximal rigid subcategories of $\bI^*_{m,n}\Int.$ 
\item Each maximal rigid subcategory $\widehat\fT$ of $\bI_{m,n}^*\Int$ has a unique maximal rigid
   consistent below (above)  completion, and exactly $2^{n-m}$  maximal rigid interval consistent completions in $\rep(\bI_{m,n}^*)$. 
\end{enumerate}
\end{thm}
 
Statement (1) originated from \cite[Lemma 4.1(2)]{LZ1}, where it was stated as a direct observation and we provide a detailed proof. In statement (3), the multiplier $2^{\,n-m}$ records two different choices of consistency conditions, below or above, for each open interval $(i,i+1)$ inside $\bI_{m,n}^*$. We also include an explicit description of all the maximal rigid interval consistent completions of a maximal rigid subcategory $\widehat\fT$ of $\bI_{m,n}^*\Int$ using the notion of ``roofs'' (see Lemma-Definition \ref{lemdef:exist1}). Thus, the classification of maximal rigid interval consistent subcategories of $\rep(\bI_{m,n}^*)$ is reduced to the classification of tilting modules over the path algebras $kQ_{m,n}^*$, which is well-known \cite{BK}.

\vskip5pt
Next, we compare various notions about tilting subcategories both in arbitrary hereditary abelian categories with enough projective objects and in specific hereditary abelian subcategories of $\rep(\mathbb R)$. 
In summary, for an arbitrary hereditary abelian category with enough projective objects, we have the following diagram of implications. 

\begin{center}
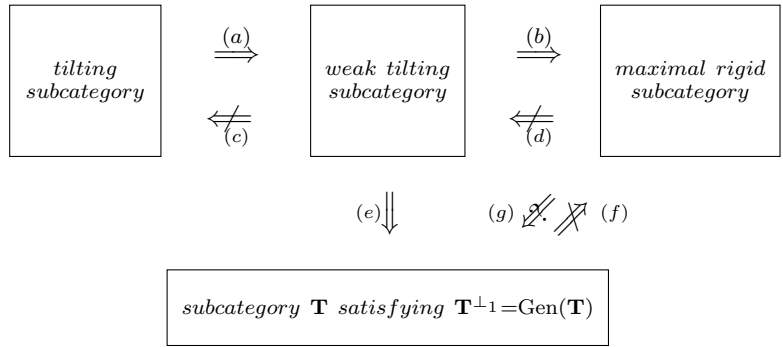
\begin{figure}[h]
\begin{tikzpicture}
\draw (-4,0) node[minimum size=2cm,draw] {$\begin{smallmatrix}tilting\\ subcategory\end{smallmatrix}$};
\draw (0,0) node[minimum size=2cm,draw]{$\begin{smallmatrix}weak\ tilting\\ subcategory\end{smallmatrix}$};
\draw (4,0) node[minimum size=2cm,draw] {$\begin{smallmatrix}maximal\ rigid\\ subcategory\end{smallmatrix}$};
\draw (0,-3) node[minimum size=1cm,draw] {$\begin{smallmatrix} subcategory\ \fT\ satisfying\ \fT^{\perp_1}=\Gen(\fT)\end{smallmatrix}$};
\draw (-2,.5) node {$\stackrel{(a)}\implies$};
\draw (-2,-0.5) node {$\not\!\!\!\Longleftarrow$};
\draw (-2,-0.75) node {\tiny$(c)$};
\draw (2,.5) node {$\stackrel{(b)}\implies$};
\draw (2,-0.5) node {$\not\!\!\!\Longleftarrow$};
\draw (2,-0.75) node {\tiny$(d)$};
\draw (0,-1.75) node[rotate=270] {$\implies$};
\draw (-0.25,-1.75) node {\tiny$(e)$};
\draw(2,-1.75) node[rotate=45]  {$\Longleftarrow$};
\draw (1.5,-1.75) node {\tiny$(g)$};\draw (2,-1.75) node[rotate=45] {$?$};
\draw(2.5,-1.75) node[rotate=45]  {$\not\!\!\! \Longrightarrow$};
\draw (3,-1.75) node {\tiny$(f)$};
\end{tikzpicture}
\caption{\centering{Relations among various notions in hereditary abelian categories with enough projective objects}}
\end{figure}
\end{center}

In this diagram, the implication $(a)$, $(e)$ and negative implication $(c)$ was shown by \cite{BBOS}. We prove the implication $(b)$ in Proposition \ref{prop:weak_tilt_rigid}, and examples are given to illustrate the negative implications $(d),(f)$ (see Example \ref{exm:max_rig_not_tilt} and \ref{exm:max_rig_not_perp1}).  
Theorem below shows that the implications $(d)$  holds  for interval consistent subcategories of $\rep(\bI_{m,n}^*)$. Moreover, the implication $(g)$ holds for any maximal rigid subcategory of $\rep(\bI_{m,n}^*)$, but fails for arbitrary  hereditary abelian category (Example \ref{exm:max_rig_not_perp1}). It is an open question whether it holds for hereditary abelian categories with enough projective objects.
 
\begin{thm}[Theorem \ref{thm:MRCB_tilt}, Proposition \ref{prop:max_rig_perp} ]\label{thm:B}The following statements hold.
\begin{enumerate}
\item Every maximal rigid interval consistent subcategory of $\rep(\bI_{m,n}^*)$ is weak tilting. 
\item Every maximal rigid consistent below subcategory of $\rep(\bI_{m,n}^*)$ is tilting.
\item For any interval $\bI\subseteq \bR$ and any $\fT$ is maximal rigid  of $\rep(\bI)$, $\fT^{\perp_1}=\Gen(\fT)$.
\end{enumerate}
\end{thm}

Statements $(1)$ and $(2)$ suggest that among all hereditary abelian categories with enough projective objects, interval consistent and consistent below subcategories of $\rep(\bI_{m,n}^*)$ have better properties with respect to tilting theory.
Notice that in statement (3), it is not necessary to assume that the maximal rigid subcategory $\fT$ is interval consistent.  In light of the implication $(b)$,  this provides a stronger statement than the implication $(e)$ for subcategories of $\rep(\bR)$.

\vskip 5pt

Finally, we study the correspondence between certain torsion classes and support tilting subcategories, which serves as a continuous version of the Ingalls-Thomas bijection (Theorem \ref{thm:IT}). The support of a subcategory is defined as the union of the supports of its objects. A subcategory is {\bf integrally supported} if its support is a finite disjoint union of integral intervals. We say an integrally supported subcategory is maximal rigidity  consistency below  if it is maximal rigidity  consistency below on each support component. A torsion class is {\bf integrally generated} if it has the form $\Gen(\widehat T)$ for an integral representation $\widehat T$.

\begin{thm}[Theorem \ref{thm:supp_MRCB}, Theorem \ref{thm:IT_contin}]\label{thm:C}
There are bijections between the following sets.
\begin{enumerate}
\item $\{$isomorphism classes of support tilting $kQ_{m,n}^*$-modules$\}$;
\item $\{$integrally supported maximal rigid consistent below subcategories of $\rep(\bI_{m,n}^*)\}$;
\item $\{$integrally generated torsion classes of $\rep(\bI_{m,n}^*)\}$.
\end{enumerate}
\end{thm}
 
\endgroup

 Moreover, finitely generated torsion classes in $\rep(\bI)$ become integrally generated after a suitable normalization. Theorem~C therefore relates finitely generated torsion classes in the continuous category to the classical support tilting theory of finite type $A$.

\subsection{Outline of the paper}
Section~2 introduces decorated real numbers and the interval notation used throughout the paper. Section~3 recalls the relevant categories
of persistence modules, defines integral and half-integral representations, and formulates the consistency conditions.
Section~4 describes rigidity through compatibility of intervals and classifies maximal rigid interval consistent subcategories, proving Theorem~A. Section~5 compares maximal rigid, weak tilting, and tilting subcategories in various settings and proves Theorem~B. Section~6 treats integral supports and integrally generated torsion classes and proves Theorem~C.

\section{Decorated real numbers and intervals on $\bR$}
Intervals on a totally ordered set $(\mathbb P,\prec)$ are  nonempty subsets of the forms:

\begin{itemize}
\item $[a,b]=\{x\in\mathbb P: a\preceq x\preceq b\},$
  $(a,b]=\{x\in\mathbb P: a\prec x\preceq b\},$
\item $[a,b)=\{x\in\mathbb P: a\preceq x \prec b\},$
  $(a,b)=\{x\in\mathbb P: a\prec x\prec b\},$
\item $(-\infty,a)=\{x\in\mathbb P:   x\prec a\},$
  $(-\infty,a]=\{x\in\mathbb P:   x\preceq a\},$
\item $(a,+\infty)=\{x\in\mathbb P: a\prec x \},$
  $[a,+\infty)=\{x\in\mathbb P: a\preceq x \},$
\end{itemize} 
for some $a,b\in \mathbb P$. Denote by $\Intv(\mathbb P)$ the set of all the intervals on $\mathbb P$.

Throughout the paper, we will use two notations for intervals on $\mathbb R$. One is the standard notation: $(a,b), (a,b], [a,b), [a,b]$. Another one using decorated real numbers $|\alpha,\beta|$. 

The set of {\bf decorated real numbers} is $\bR^{\pm}=\{x,x^+,x^-:x\in\bR\}$. The total order on $\bR^\pm$ is defined by (i) $x^\pm\prec y^\pm$ if $x<y$, (ii) $x^-\prec x\prec x^+$. Usually, we denote a decorated number by a Greek letter. 
Recall that an immediate successor of an element $\alpha$ is an element $\beta\succ\alpha$ and there is no element $\beta'$ such that $\beta\succ\beta'\succ\alpha$. Denote the immediate successor of $\alpha$ by $\suc(\alpha)$. Similarly, we can define the immediate predecessor $\pre(\alpha)$ of an element $\alpha$. As a result, a decorated  number $\alpha$ has an  immediate successor $\suc(\alpha)$ if and only if $\alpha=x$ or $x^-$ for some real number $x$ and $\alpha$ has an immediate predecessor $\pre(\alpha)$ if and only if $\alpha=x$ or $x^+$ for some real number $x$. The {\bf real value} of a decorated number $\alpha$ is a real number denoted by $\|\alpha\|$ and $\|\alpha\|=x$ if $\alpha=x$ or $x^+$ or $x^-$ for some real number $x$. We consider $\bR\subseteq\bR^\pm$ as a subset. The following proposition is an easy generalization of a well-known fact about real numbers.

\begin{prop}
    Any  subset of $\bR^\pm$ bounded below (resp. above) has a infimum (resp. supremum).
\end{prop}
\begin{proof}
    Suppose $A\subseteq \mathbb R^\pm$ is bounded below.   Then the set $B=\{\|\alpha\|: \alpha \in A\}$ is a bounded subset of $\mathbb R$. Hence $B$ has an infimum $b$. 
    Then a case by case test leads to the consequence:
    $$
   \inf A=\begin{cases}
    b^- & b^-\in A\\
    b  & b\in A \text{ and } b^-\notin A\\
    b^+ & \text{else}.
    \end{cases}
    $$
 The statement for subsets bounded above can be proved similarly.
\end{proof}

For the sake of defining unbounded intervals, we also need to add a minimum $-\infty$ and maximum $+\infty$ to the totally ordered set $\bR$ and call $\bR\cup\{\pm\infty\}$ the set of extended real numbers. Similarly, by adding a minimum and maximum to the decorated real numbers, we obtain the set of extended real numbers $\bR^{\pm}\cup\{\pm\infty\}$. By convention, define the real values $\|-\infty\|=-\infty$ and $\|+\infty\|=+\infty$.

\vskip5pt
{\bf Decorated  notation of intervals on $\bR$.} For $\alpha,\beta\in\bR^{\pm}\cup\{\pm\infty\}$ such that $+\infty\neq\alpha\preceq\beta\neq-\infty$, denote by  $|\alpha,\beta|:=\{\gamma\in\bR:\alpha\preceq\gamma\preceq\beta\}$ an interval on $\mathbb R$. We say an interval $|\alpha,\beta|$ is {\bf admissible} if $\alpha\neq x^-$ and $\beta\neq y^+$ for any real numbers $x,y$.  
 As shown in the table below, there is a one-to-one correspondence between the admissible intervals in decorated numbers notation and standard interval notation. 
\begin{table}[h] 
 \begin{center}
 \tiny{
 \begin{tabular}{|p{1.3cm}|c|c|c|c|c|c|c|c|c|}
 \hline
 Decorated notation& $|x,y|$ & $|x^+,y|$ & $|x,y^-|$ & $|x^+,y^-|$ & $|-\infty,y|$ &$|-\infty,y^-|$&$|x,+\infty|$& $|x^+,+\infty|$ & $|-\infty,+\infty|$\\  
 \hline
 Standard notation& $[x,y]$ & $(x,y]$ & $[x,y)$& $(x,y)$& $(-\infty,y]$& $(-\infty,y)$& $[x,+\infty)$& $(x,+\infty)$& $(-\infty,\infty)$\\
 \hline
 \end{tabular}}
 \caption{The correspondence between decorated and standard notations}
 \end{center}
\end{table}
In the following, without specific mention, we assume all the intervals are admissible. One should be comfortable to apply the decorated real numbers notation whenever it is convenient. 

 \section{Persistence modules}

 \subsection{Representations over totally ordered sets}
Let $\bP$ be a totally ordered set and $k$ a fixed field. A persistence $\bP$-module (representation) $M$ is a functor from $\bP$ to $k$-vector spaces. Precisely, $M(x)$ is a vector space for each $x\in \bP$ and $\rho_{xy}:M(x)\to M(y)$ is a linear map for each pair $x\leq y$, satisfying $\rho_{yz}\rho_{xy}=\rho_{xz}$ for $x\leq y\leq z$ and $\rho_{xx}=1_{M(x)}$ for each $x\in \bP$. A morphism $f:M\to N$ between persistence modules consists of pointwise linear maps $M(x)\to N(x)$ such that the following diagram commutes: 
$$
\xymatrix{M(x)\ar[r]^{\rho_{xy}}\ar[d]^{f_x} &M(y)\ar[d]^{f_y} \\ N(x)\ar[r]^{\rho'_{xy}} &N(y).}
$$ 
Denote by $\Rep(\bP)$ the category of  persistence $\bP$-modules.

The support of $M$ is the set $x$ such that $M(x)\neq 0$, denoted by $\supp(M)$. 
Persistence modules supported on intervals are called ``interval modules''.
For a nonempty interval $I\in\Intv(\mathbb P)$, denote by  $\bk_I$ the {\bf interval module} on $I$, which is precisely defined by
\begin{align*}
    \bk_{I}(x) &= \begin{cases}
        k & x\in I\\
        0 & \text{otherwise}
    \end{cases}
\end{align*}
and  $\rho_{xy}  =  1$  for $x\leq y, x,y\in I $. 


A persistence $\bP$-module $M$ is called pointwise finite dimensional, if $\dim_kM_x<\infty$ for each $x\in \bP$. Denote by $\bP\pwf$ the category of pointwise finite dimensional persistence $\bP$-modules. It turns out that $\bP\pwf$ is an  abelian category. 
  The following theorem shows that $\bP\pwf$ is a Krull-Schmidt-Remark-Azumaya category.

\begin{thm}\cite[Theorem 1.1]{CB}\label{thm:barcode ]}
  Let $\bP$ be a totally ordered set.  Then any pointwise finite-dimensional  $\bP$-module decomposes uniquely up to isomorphism into a direct sum of interval modules. 
\end{thm}

We are mostly interested in persistence modules of $\bR$. Indecomposable projective objects in $\bR\pwf$ are $k_{[a,+\infty)}$, $k_{(a,+\infty)}$ or $k_{(-\infty,+\infty)}$. A persistence module $M$ is called {\bf finitely generated} if it is a quotient of a finite direct sum of indecomposable projective objects in $\bR\pwf$. We define $\rep(\bR)$ as the full subcategory of $\bR\pwf$ whose objects are
finitely generated persistence $\bR$-modules. It  turns out that $\rep(\bR)$ is an Krull-Schmidt abelian subcategory of $\bR\pwf$ consisting of objects which are finite direct sums of interval modules \cite[Lemma 3.1.4-3.1.5]{IRT}. For an interval $\mathbb I\subseteq\bR$, denote by $\rep(\mathbb I)$ the full subcategory of finitely generated modules whose supports are contained in $\mathbb I$, which are also Krull-Schmidt abelian categories. 

\begin{rmk}
Both $\bR\pwf$ and $\rep(\bR)$ are hereditary. The category $\bR\pwf$ does not have enough projective objects. For example, $M=\bigoplus\limits_{x\in[0,1]}k_{[x,x]}$ is pointwise finite dimensional, but it does not admit a projective cover in $\bR\pwf$. In contrast, each object in $\rep(\bR)$ has a projective cover, due to the finiteness of direct sums. 
\end{rmk}

Recall that an order preserving (resp. reversing) homeomorphism $f:\bR\to \bR$ is a continuous increasing function. Denoted by $\Homeo_+(\bR)$ the set of order preserving  homeomorphisms. The following result is well-known.

\begin{prop}\label{prop:homeo}
Let $\mathbb I,\mathbb J$ be intervals on $\bR$.  If there is a map $f\in\Homeo_+(\bR)$ such that $f(\bI)=\mathbb J$, then there is an equivalence  of abelian categories $F:\rep(\mathbb I)\cong \rep(\mathbb J)$.
\end{prop}
\begin{proof}
   Let $M=(M(x), \rho_{xy})$ be a persistence $\bI$-module. It is easy to check that $FM=(FM(x),\rho'_{xy})$, given by $FM(x)=M(f(x)$ and $\rho'_{xy}=\rho_{f(x)f(y)}$ defines an equivalence.
\end{proof}

Therefore,  the category of finitely generated modules over any interval on $\bR$ is equivalent to one of the following:
$
(1) \rep([a,b]) , (2) \rep([a,b)), (3) \rep((a,b]), (4) \rep((a,b)),
$
for some real numbers $a\leq b$.

\subsection{Integral and half-integral representations}
In practice, we can only record the value of persistence modules on a certain set of critical points instead of the whole set $\bP$. Restricting to interval modules born or dead at those points gives rise to the concept of (half-)integral representations, which also allow us to apply tools from finite quiver representations conveniently.

\begin{defn} 
Let $\bI=|\alpha,\beta|$ be a (possibly unbounded) interval on $\mathbb R$. A {\bf critical vector} $\bm\lambda=(\lambda_0, \cdots, \lambda_n)$ on $\bI$ consists of a finite sequence of extended decorated numbers $\alpha=\lambda_0\prec\lambda_1\prec\cdots\prec\lambda_n=\beta$, where $\lambda_1,\cdots,\lambda_{n-1}\in\bR$ are   distinct from $\|\lambda_0\|$ and $\|\lambda_n\|$. By convention, the critical vector for intervals $\bI=[x,x]$ is $\bm\lambda=(x)$. 

Given a critical vector $\bm\lambda$, an admissible interval $|\mu,\nu|\subseteq\bI$ is called {\bf integral} if both $\|\mu\|$ and $\|\nu\|\in\{\|\lambda_i\|:i=0,1,\cdots,n\}$ and {\bf half-integral} if either $\|\mu\|$ or $\|\nu\|\in\{\|\lambda_i\|:i=0,1,\cdots,n\}$

A {\bf finitely generated (half-)integral representation} over $(\bI,\bm\lambda)$ is a finitely generated representation $M\cong\bigoplus \bk_{I_i}$, such that $I_i$'s are (half-)integral intervals.  
 Denote by $(\bI,\bm\lambda)\Int$ the full subcategory of $\rep(\bI)$ consisting of integral representations.
\end{defn}

\begin{exm}
 Let $\bI=(0,1]$ and $\bm\lambda=(\lambda_0,\lambda_1)=(0^+,1)$. 
There are $3$ indecomposable integral representations in $(\bI,\bm\lambda)\Int${\rm:} $\bk_{(0,1]}$, $\bk_{(0,1)}$ and $\bk_{[1,1]}$. A complete set of non-isomorphic indecmoposable half-integral representations is 
$$
\{\bk_{(0,x)}, \bk_{(0,x]}: 0<x\leq 1\}\cup\{\bk_{(y,1)},\bk_{[y,1)}:0<y<1\}\cup\{\bk_{[z,1]},\bk_{(z,1]}:0<z<1\}\cup\{k_{[1,1]}\}.
$$

\end{exm}

In light of Proposition \ref{prop:homeo}, we have the following observation.
\begin{prop}\label{prop:homeo_int}
    Let $\bI=|\alpha,\beta|$ be an interval on $\bR$ with a critical vector $\bm\lambda=(\lambda_0,\lambda_1,\cdots,\lambda_n)$. 
   If there is a map $f\in\Homeo_+(\bR)$  such that $f(\bI)=\mathbb J=|\gamma,\delta|$ and $\bm\lambda'=(\gamma, f(\lambda_1), \cdots,$ $ f(\lambda_{n-1}),\delta)$,
    then there is an equivalence of abelian categories $(\bI,\bm\lambda)\Int\cong(\mathbb J,\bm\lambda')\Int$.
\end{prop}

One can always normalize the interval $\bI$ and the critical vector $\bm\lambda$ using some order-preserving homeomorphism $f$, such that $f(\bI)=|m^{(+)}, n^{(-)}|$ for some integers $m, n$, and $\lambda_i=m+i$ for $i=1,\cdots,\mbox{$n-m-1$}$. It leads to our definition of continuous quiver of type $\mathbb A$.

\begin{defn}\label{def:RAn}
 For $n\geq 1$, denote by $\bI_{m,n}=[m,n]$, $\bI_{m,n}^{+}=(m,n]$, $\bI_{m,n}^-=[m,n)$, $\bI_{m,n}^\pm=(m,n)$.
Each $\bI_{m,n}^*$ ($*=\emptyset,+,-,\pm$), is associated with a critical vector $\bm\lambda^*=(\lambda_0,\lambda_1,\cdots,\lambda_{n-1},\lambda_n)$, where $\lambda_i=m+i$, for $1\leq i\leq n-m-1$. The totally ordered set $\bI_{m,n}^*$ together with the critical vector $\bm\lambda^*$ is called a {\bf continuous quiver of type $\bA$}. Simply denote 
the category of integral representations of $\bI_{m,n}^*$ by $\bI_{m,n}^*\Int$.
\end{defn}

\begin{rmk}\ \\
{\rm(1)} The category of finitely generated modules over $\bI_{m,n}^*$ is nothing but finitely generated modules over the corresponding intervals. The critical vector only plays a role, when we consider the (half-)integral representations.\\
{\rm(2)} As a consequence of Proposition \ref{prop:homeo_int}, for any interval $\bI$ and critical vector $\bm\lambda$, the category $(\bI,\bm\lambda)\Int$ is equivalent to the category of  integral representations over a continuous quiver of type $\bA$.\\
{\rm(3)}  In general, it is possible to consider an infinite set of critical points, but with some technical obstacles of dealing with accumulation points, which is slightly off-topic. In this paper, we aim to relate continuous quivers with finite quivers. Therefore, only  finite sets of critical points are taken into account.

\end{rmk}
 
The important link between integral representations and modules over path algebras is due to the following observation.
Denote by $\mathbb Z^\sharp$ the totally ordered set of extended integers:  $$\cdots<-1<(-1)^\sharp<0<0^\sharp<1<1^\sharp<\cdots.$$
Then there is a bijection $\theta:\Intv(\mathbb Z^\sharp)\to \Intv_{\mathbb Z}(\bR)$ given below:
\begin{align*}
   \theta([i,j])=[i,j], \ \ 
   \theta([i^\sharp,j])= (i,j], \ \ 
   \theta([i,(j-1)^\sharp])= [i,j), \ \ 
   \theta([i^\sharp,(j-1)^\sharp])=(i,j).
\end{align*}
 
Now for any integral interval $\bI_{m,n}^*$ on $\bR$, $\theta^{-1}(\bI_{m,n}^*)$ is an interval on $\mathbb Z^\sharp$, whose incidence algebra is a path algebra of a linear type $A$ quiver, denoted by $Q_{m,n}^*$. Precisely,
\begin{align}\label{eqn:Q_n}
&Q_{m,n}:m\to m^\sharp \to m+1\to (m+1)^\sharp \to\cdots \to (n-1)^\sharp\to n, \nonumber\\
&Q_{m,n}^+:  m^\sharp \to m+1\to (m+1)^\sharp\to\cdots \to (n-1)^\sharp\to n,  \nonumber\\
&Q_{m,n}^-:m\to m^\sharp \to m+1\to (m+1)^\sharp \to\cdots \to (n-1)^\sharp,  \nonumber \\
&Q_{m,n}^{\pm}:  m^\sharp \to m+1\to (m+1)^\sharp \to\cdots \to (n-1)^\sharp.
\end{align}
Indecomposable $kQ_{m,n}^*$-modules can also be viewed as interval modules over the poset $\mathbb Z^\sharp$. That is for each interval $I$, denote by $k_I$ the indecomposable $kQ_{m,n}^*$-module  supported on the vertices belongs to $I$. 
From the bijection $\theta$, the following result follows immediately.
 \begin{lem}\label{lem:eqv_A}
 There are equivalences of abelian categories $\fD_{m,n}^*: kQ_{m,n}^*\modu\to\bI_{m,n}^*\Int$, given by $\fD_{m,n}^*(k_I)=k_{\theta(I)}$.
\end{lem}  
  
\subsection{Consistency conditions}
In next section, we will discuss maximal rigid subcategories in $\rep(\bR)$. A typical class of maximal rigid subcategories satisfies the so-called consistency conditions, which we will define in this subsection.

 \begin{defn}\label{def:int_consist}

Let $\bI$ be an interval on $\bR$ with a   critical vector $\bm\lambda=(\lambda_0,\lambda_1,\cdots,\lambda_{n-1},\lambda_n)$. A  set of half-integral intervals $\cM$ on $\bI$ is called a {\bf lower bound consistent covering} of the open interval $(\|\lambda_i\|,\|\lambda_{i+1}\|)$ if there exists a unique decorated real number $\alpha_i$, such that 
\begin{enumerate}
\item[(C1)] for each $x\in(\|\lambda_i\|,\|\lambda_{i+1}\|)$, $|\alpha_i,x|, |\alpha_i,x^-|\in\cM$;
\item[(C2)] if $|\alpha,x|\in\cM$ for some $x\in(\|\lambda_i\|,\|\lambda_{i+1}\|)$, then $\alpha=\alpha_i$;
\item[(C3)] for each $x\in(\|\lambda_i\|,\|\lambda_{i+1}\|)$, $\cM$ does not contain any intervals $|x,\beta|$ or $|x^+,\beta|$.
\end{enumerate} 
And $\cM$ is called a {\bf upper bound consistent covering} of the open interval $(\|\lambda_i\|,\|\lambda_{i+1}\|)$ if there exists a unique decorated real number $\beta_i$, such that\begin{enumerate}
\item[(C1')] for each $x\in(\|\lambda_i\|,\|\lambda_{i+1}\|)$, $|x,\beta_i|, |x^+,\beta_i|\in\cM$;
\item[(C2')] if $|x,\beta|\in\cM$ for some $x\in(\|\lambda_i\|,\|\lambda_{i+1}\|)$, then $\beta=\beta_i$;
\item[(C3')] for each $x\in(\|\lambda_i\|,\|\lambda_{i+1}\|)$, $\cM$ does not contain any intervals  $|\alpha_i,x|, |\alpha_i,x^-|$.
\end{enumerate} 

 We say $\cM$ is {\bf consistent above/below}  if $\cM$ is an  upper/lower bound consistent covering for each open interval $(\|\lambda_i\|,\|\lambda_{i+1}\|)\subseteq \bI$. We say $\cM$ is {\bf interval consistent} if for each open interval $(\|\lambda_i\|,\|\lambda_{i+1}\|)\subseteq I$, it is either a lower or upper bound consistent covering. We call an additive subcategory $\fT\subseteq\rep(\bI)$ interval consistent  (resp. consistent above/below)  if $\cM=\{I:k_I\in\fT\}$ is interval consistent (resp. consistent above/below). 
 
 For an interval consistent set $\cM$, denote by $\widehat\cM$ the largest subset of $\cM$ consisting of integral intervals. And denote by $\widehat\fT$ the largest full subcategory of an interval consist subcategory $\fT$ consisting of integral representations. 
 \end{defn}

\begin{rmk}\ \\
{\rm (1)} The concept of interval consistent subcategoires over the interval $\bI=[0,1]$ was introduced in \cite{LZ1}, where it was called representations of type $\bm\lambda$. \\
{\rm (2)} By definition, consistent above/below sets are interval consistent. And the interval consistency of $\cM$ is a condition imposed on $\cM\setminus \widehat\cM$. Whether  $\cM$ contains integral intervals does not affect its interval consistency property.
\end{rmk}

\begin{exm}
 Let $\bI=[0,+\infty)$ and $\bm\lambda=(\lambda_0,\lambda_1,\lambda_2)=(0,1,+\infty)$.

(1) Let $\fT$ be an additive full subcategory of $\bI\pwf$ consists of indecomposable modules:  
$\bk_{(0,x)}, \bk_{(0,x]}$, for $0<x\leq 1$  and  
$\bk_{(x,+\infty)}, \bk_{[x,+\infty)}$, for $1\leq x<+\infty. $
Then $\fT$ is interval consistent.

(2) Let $\fT'$ be an additive full subcategory of $\bI\pwf$ consists of indecomposable modules:  
$\bk_{(0,x)}, \bk_{(0,x]}, \bk_{[0,x]}$, for $0<x\leq 1$  and  
$\bk_{[x,+\infty)}$, for $1\leq x<+\infty. $
Then $\fT$ is not consistent. Because the $\cM$ contains both $\bk_{[0,x]}$ and $\bk_{(0,x]}$  and misses modules $\bk_{(x,+\infty)}$ for $1<x<+\infty$.
\end{exm}

We can specialize the notion of interval consistent subcategories for continuous quivers of type $\bA$. Let $\bI_{m,n}^*$ be a continuous quiver as in Definition \ref{def:RAn}. Then a set of half-integral intervals $\cM$ on $\bI_{m,n}^*$ is consistent above/below if it is a upper/lower bound consistent covering for each open interval for each interval $(i,i+1)$, $i=m,\cdots,n-1$.
And it is interval consistent if it is an upper or lower bound consistent covering for each interval $(i,i+1)$, $i=m,\cdots,n-1$. 

\section{Maximal rigid interval consistent subcategories}

\subsection{Rigid subcategories and compatible sets}
Our main goal in this section is to give a full characterization of maximal rigid subcategories of the category of continuous representations of type $A$, which are interval consistent. First we recall the definition of maximal rigid subcategories of an abelian category.
 \begin{defn}\label{def:rigid_ab}
 Let $\cA$ be an  abelian category.  An additive full subcategory $\fT\subseteq \cA$ is called rigid if  $\Ext^1_{\cA}(M,N)=0$ for any pair of objects $M,N\in \fT$. It is called maximal rigid, if $\fT$ is not a proper subcategory of any rigid subcategory. Denote by $\Rig(\cA)$ (resp. $\MRig(\cA)$) the collection of (maximal) rigid subcategories of $\cA$.
 \end{defn}

Recall that a category is skeletal small if the collection of isomorphism classes of objects forms a small set. The collection of additive subcategories of a skeletal small abelian category is still a small set. So whenever $\cA$ is skeletal small, $\Rig(\cA)$ and $\MRig(\cA)$ are actually sets rather than proper classes. 
As far as (maximal) rigid subcategories of abelian subcategories of $\rep(\bR)$ are considered, it is essential to study the intervals $I$ corresponding to the indecomposable interval modules $k_I$. So we introduce the notion of compatible sets.

 \begin{defn}\label{def:compatible}
Two intervals $I$ and $J$ (or two interval modules $k_I$ and $k_J$) are {\bf mutually compatible} if $\Ext^1(\bk_I,\bk_J)=0=\Ext^1(\bk_J,\bk_I)$.
An  interval set $\mathcal M$ is called {\bf compatible} if $\forall I, J\in \cM$, $I$ and $J$ are mutually compatible. A subset $\cM$ of an interval set $\mathcal I$ is called maximal compatible in $\mathcal I$ if $\cM$ is not a proper subset of any compatible subset of $\mathcal I$. 
 \end{defn}

The following lemma will be heavily used later, where the first two statements give algebraic and geometric characterizations of mutually compatible intervals respectively, and  the last statement  is useful for inductive arguments. 

 \begin{lem}\label{lem:key} The following statements hold.
 \begin{enumerate}
\item Let $|\alpha,\beta|$ and $|\gamma,\delta|$ be admissible intervals in $\bR$. Then $\Ext^1(\bk_{|\alpha,\beta|},\bk_{|\gamma,\delta|})\neq0$ if and only if $\alpha\prec \gamma \preceq \suc(\beta) \preceq \delta$. 
\item Let $I$ and $J$ be two intervals in $\bR$. Then $\Ext^1(\bk_I,\bk_J)=0=\Ext^1(\bk_J,\bk_I)$ if and only if either there are containments $I\subseteq J$ or $J\subseteq I$, or $I\cup J$ is not an interval.
\item Let $\mathcal I$ be a set of intervals and $\cM$ be a compatible subset of $\mathcal I$. For an interval $I\in \mathcal M$, denote by $\cM_I=\{J\in \cM: J\subseteq I\}$. If an interval $K\subseteq I$ is compatible with all intervals in $\cM_I$, then it is compatible with all intervals in $\cM$.
\end{enumerate}
\end{lem}
\begin{proof}
(1) Nonzero extensions between $\bk_{|\alpha,\beta|}$ and $\bk_{|\gamma,\delta|}$ correspond to short exact sequences:
\begin{align*}
&0\to \bk_{|\gamma,\delta|} \to \bk_{|\alpha,\delta|}\oplus \bk_{|\gamma,\beta|}\to \bk_{|\alpha,\beta|}\to 0, \text{\ when } \alpha\prec\gamma\prec \beta\prec\delta,\\
\text{or\ \ } &0\to \bk_{|\gamma,\delta|} \to \bk_{|\alpha,\delta|} \to \bk_{|\alpha,\beta|}\to 0, \text{\ when } \gamma =\suc(\beta).
\end{align*}
(2) follows from (1).\\
(3) Denote by ${\cM_I}'=\{J\in\cM: I\subseteq J\}$ and ${\cM_I}''=\{J\in\cM: I\cup J$ is not an interval$\}$. According to Lemma \ref{lem:key}, $\cM=\cM_I\cup{\cM_I}'\cup {\cM_I}''$. It is clear that $K$ is compatible with intervals in ${\cM_I}'$ and ${\cM_I}''$. So it is compatible with all intervals in $\cM$.
\end{proof}

Let $\Intv(\bI_{m,n}^*)$ (resp. $\widehat\Intv(\bI_{m,n}^*)$) be the set of all (resp. integral) intervals on $\bI_{m,n}^*$. We say $\cM$ is a {\bf (maximal) compatible set on $\bI_{m,n}^*$} if it is a (maximal) compatible subset of $\Intv(\bI_{m,n}^*)$. We say $\cM$ is a {\bf (maximal) compatible set of integral intervals on $\bI_{m,n}^*$}, if it is a (maximal) compatible subset of $\widehat\Intv(\bI_{m,n}^*)$.
 Denoted by $\Com(\bI_{m,n}^*)$ (resp. $\MCom(\bI_{m,n}^*)$) the set of all (resp. maximal) compatible sets on $\bI_{m,n}^*$ and $\widehat\Com(\bI_{m,n}^*)$ (resp. $\widehat\MCom(\bI_{m,n}^*)$) the set of all (resp. maximal) compatible sets of integral intervals on $\bI_{m,n}^*$.

For a subcategory $\fT$ of $\rep(\bR)$, denote by $\mathcal S(\fT)=\{I:k_I\in\fT\}$. From the definition, ${\mathcal S}$ defines a function between the set of rigid subcategories and compatible sets. 

\begin{prop}\label{prop:supp}
The map ${\mathcal S}$ yields bijections between the following sets:
 \begin{align*}
{\mathcal S}:  
  \MRig (\rep(\bI_{m,n}^*)) &\leftrightarrow \MCom(\bI_{m,n}^*) \\
  \MRig (\bI_{m,n}^*\Int) &\leftrightarrow \widehat\MCom(\bI_{m,n}^*) \\
\fT &\mapsto  {\mathcal S}(\fT)  \\
\fT=\add\{  k_I  :I\in\cM\} &\mapsfrom \cM 
 \end{align*}
 \end{prop}
  Thus, classifying maximal rigid subcategories is equivalent to classifying maximal compatible sets. We start from several motivating examples.
 
 \begin{exm}\label{exm:2}  
The set $\mathcal M=\{ [x,2], (x,2] : 0\leq x\leq 2\}$ is maximal compatible on $[0,2]$.

In fact, $\cM$ is clearly compatible. Assume $\add (\cM\cup\{|\alpha,\beta|\})$ is also compatible. If $\alpha=\beta$, then they must be a real number. Hence $|\alpha,\alpha|$ is compatible with $\cM$ if and only if $\alpha=2$ and then $|\alpha,\alpha|\in\cM$. If $\alpha\prec\beta\prec 2$,   there exists real number $x$, such that $\alpha\preceq x\preceq \beta$. Hence  $\alpha\prec  \suc(x)\preceq \suc(\beta) \preceq 2$, which implies that  $|\alpha,\beta|$ is incompatible with $(x,2]$, a contradiction. Therefore $\beta=2$ and $|\alpha,\beta|\in\cM$. 
\end{exm} 

\begin{exm}\label{exm:1}
Let $\cM=\cM_1\cup\cM_2$, where   
   $\cM_1=\{\displaystyle(\frac{m}{2^n},\frac{m+1}{2^n}): n\geq 0, 0\leq m\leq 2^n-1\}, $ and
  $\cM_2=\{[x,x]: x\in \mathbb R, x\neq\displaystyle\frac{m}{2^n}, n\geq 0, 0<m<2^n-1\}.$ Then we claim that $\cM$ is a maximal compatible set on $(0,1)$.
  
  \begin{figure}[h]
    \centering
    \begin{tikzpicture}
         \draw[-{To[scale=1.3]},line width=1pt, color=black] (-3.4, 0) -- (3.4, 0);
    \foreach \x in {-3,-1.5,...,3}
        \draw[shift={(\x,0)},color=black, line width=1pt] (0pt,0pt) -- (0pt,5pt);
    \draw[shift={(-3,-0.4)},color=black] node[font=\large,text height=12pt] {$0$};
    \draw[shift={(3,-0.4)},color=black] node[font=\large,text height=12pt] {$1$};
 \draw[shift={(0,-0.4)},color=black] node[font=\large,text height=12pt] {$\frac{1}{2}$};
 \draw[shift={(-1.5,-0.4)},color=black] node[font=\large,text height=12pt] {$\frac{1}{4}$};
 \draw[shift={(1.5,-0.4)},color=black] node[font=\large,text height=12pt] {$\frac{3}{4}$};
\draw[shift={(-2.25,.25)},color=black] node[font=\small,text height=12pt] {$\cdots$};
 \draw[shift={(-0.75,.25)},color=black] node[font=\small,text height=12pt] {$\cdots$};
 \draw[shift={(0.75,.25)},color=black] node[font=\small,text height=12pt] {$\cdots$};
 \draw[shift={(2.25,.25)},color=black] node[font=\small,text height=12pt] {$\cdots$};
  \draw (-3,0)  .. controls  +(north:18 mm) and +(north:18 mm) .. (3,0);
\draw (-3,0)  .. controls  +(north:8 mm) and +(north:8 mm) ..   (0,0);
\draw (0,0)  .. controls  +(north:8 mm) and +(north:8 mm) ..   (3,0);
\draw (-3,0)  .. controls  +(north:4 mm) and +(north:4 mm) ..   (-1.5,0);
\draw (-1.5,0)  .. controls  +(north:4 mm) and +(north:4 mm) ..   (0,0);
\draw (0,0)  .. controls  +(north:4 mm) and +(north:4 mm) ..   (1.5,0);
\draw (1.5,0)  .. controls  +(north:4 mm) and +(north:4 mm) ..   (3,0);
    \end{tikzpicture}
 \caption{An example of a maximal compatible set}
    \label{fig:exm_1}
\end{figure}
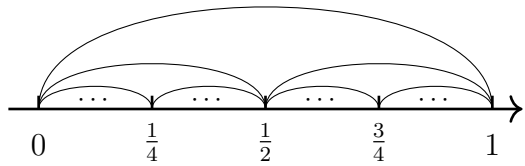

  In fact, it is easy to check that $\cM$ is compatible.  Now assume $\add (\cM\cup\{|\alpha,\beta|\})$ is also compatible. If $\|\alpha\|=\|\beta\|$, then $|\alpha,\beta|$ must be a closed interval of the form $[x,x]$. The compatibility forces $\alpha=\beta=x\in\bR\setminus\{\displaystyle\frac{m}{2^n}: n\geq 0, 0<m<2^n-1 \}$. Hence $|\alpha,\beta|\in\cM$.
  
  In the following, we assume $\|\alpha\|\neq\|\beta\|$.  If $\|\alpha\|\neq\displaystyle\frac{m}{2^n}$ for $n\geq 0, 0<m<2^n-1$, then there exists a sufficiently large $N$ and $m$, such that $\displaystyle\frac{m}{2^N}\prec\alpha\prec \frac{m+1}{2^N}\prec \beta$, which is a contradiction. Hence $\|\alpha\|=\displaystyle\frac{m}{2^n}$ for some integer $m,n$. For the same reason, $\|\beta\|=\displaystyle\frac{p}{2^q}$ for some integer $p,q$.
  Since the set $\cM_1$ contains all open intervals $\displaystyle(\frac{m}{2^n},\frac{m+1}{2^n})$. This  forces $|\alpha,\beta|$ to be an open interval $(\displaystyle\frac{m}{2^n},\frac{p}{2^q})$ as well. Assuming both fractions  $\displaystyle\frac{m}{2^n}$ and $\displaystyle\frac{p}{2^q}$ are reduced, i.e. $m,p$ are odd integers, to show $|\alpha,\beta|\in\cM_1$, it suffices to prove that $\displaystyle\frac{p}{2^q}-\frac{m}{2^n}=\frac{1}{2^{\max\{n,q\}}}$.\\
  
  Suppose $q\leq n$. Since $m$ is an odd integer, from the construction of $\cM_1$, it follows that $\displaystyle(\frac{m-1}{2^n},\frac{m+1}{2^n})\in\cM_1$. So if  $\displaystyle\frac{p}{2^q}-\frac{m}{2^n}>\frac{1}{2^n}$, the intervals $\displaystyle(\frac{m-1}{2^n},\frac{m+1}{2^n})$ and $(\displaystyle\frac{m}{2^n},\frac{p}{2^q})$ are not mutually compatible, which is a contradiction. Meanwhile, obviously we have $\displaystyle\frac{p}{2^q}-\frac{m}{2^n}\geq\frac{1}{2^n}$. Thus, it forces $\displaystyle\frac{p}{2^q}-\frac{m}{2^n}=\frac{1}{2^n}$. A similar argument applies for the case $n\geq q$. Therefore, $|\alpha,\beta|\in\cM$ and $\cM$ is a maximal rigid set.
 \end{exm} 
   
In Example \ref{exm:2}, one can clearly see its analogue with tilting modules: the rigid subcategory $\fT=\add\{k_I:I\in\cM\}$ contains exactly all indecomposable projective objects. In contrast, Example \ref{exm:1} produce a maximal compatible set in a different flavor, suggesting that not all maximal rigid subcategories are interval consistent for some critical vector. So it seems ambitious to look for a classification of all maximal compatible sets on $\bR$.
One should also compare the second example with so-called ``standard cluster'' in continuous cluster categories (see \cite[Definition 4.2.6]{IT1}). The rest of this section is devoted to study maximal compatible sets similar to the first example, which are in fact interval consistent. 

\subsection{Maximal rigid interval consistent subcategories}
A subcategory $\fT$ of $\rep(\bI_{m,n}^*)$ is called rigid interval consistent if $\fT$ is both rigid and interval consistent. A rigid interval consistent subcategory is called {\bf maximal rigid interval consistent (MRIC)} if it is not a proper subcategory of any  rigid interval consistent subcategories. Denote by $\MRIC(\bI_{m,n}^*)$ the set of all MRIC subcategories of $\rep(\bI_{m,n}^*)$.
 Accordingly, a set $\cM$ of semi-integral intervals on $\bI_{m,n}^*$ is called compatible interval consistent if it is both compatible and interval consistent. A compatible interval consistent set is called  {\bf maximal compatible interval consistent (MCIC)}, if it is not a proper subset of any  compatible interval consistent set. Denote by $\MCIC(\bI_{m,n}^*)$ the set of all MCIC sets on $\bI_{m,n}^*$.
Similarly, one can define {\bf maximal rigid consistent below/above (MRCB/MRCA)} subcategories and {\bf maximal compatible consistent below/above (MCCB/MCCA)} sets. Denote by $\MRCB(\bI_{m,n}^*)$ (resp. $\MRCA(\bI_{m,n}^*)$) the set of all MRCB (resp. MRCA) subcategories of $\rep(\bI_{m,n}^*)$.
 
Notice that from the definition,  MCIC subsets on $\bI_{m,n}^*$ are only maximal  compatible among interval consistent subsets  of $\Intv(\bI_{m,n}^*)$. The following proposition shows that they are actually maximal compatible among all subsets. 

 \begin{lem} \label{lem:max_rig_ic}\ 
 Let $\cM$ be an interval consistent set on $\bI_{m,n}^*$. If a non-integral interval $J=|\alpha,\beta|$ is  compatible with all intervals in $\cM$, then $J\in\cM$. 
 Consequently, each MCIC set on $\bI_{m,n}^*$ is maximal compatible on $\bI_{m,n}^*$.
\end{lem}
\begin{proof} 
Without loss of generality, assume $\|\alpha\|$ is not an integer and $i^+\prec\alpha\prec (i+1)^-$. First, notice that $\beta\succeq (i+1)^-$. Otherwise, it wouldn't be possible for any semi-integral interval starting or ending at $\alpha\prec x\prec \beta$ to be compatible with  $J$. In order to be compatible with $J$, the interval consistency condition implies that $\cM$ must contain semi-integral intervals $|x,\gamma|$ rather than  $|\gamma,x|$ for each real number $i<x<i+1$,  which also forces $\gamma=\beta$. But this means $J\in\cM$.

\begin{figure}[h]
    \centering
    \begin{tikzpicture}
         \draw[-{To[scale=1.3]},line width=1pt, color=black] (-3.4, 0) -- (3.4, 0);
    \foreach \x in {-3,-1,...,3}
        \draw[shift={(\x,0)},color=black, line width=1pt] (0pt,0pt) -- (0pt,5pt);
    \draw[shift={(-3,-0.4)},color=black] node[font=\large,text height=12pt] {$i$};
    \draw[shift={(-1,-0.4)},color=black] node[font=\large,text height=12pt] {$i+1$};
    \draw[shift={(1,-0.4)},color=black] node[font=\large,text height=12pt] {$\beta$};
 \draw[shift={(-2,-0.4)},color=black] node[font=\large,text height=12pt] (a){$\alpha$};

  \draw (-2,0)  .. controls  +(north:10 mm) and +(north:10 mm) .. node[above=1pt,text height=14pt,font=\small]{$J$} (1,0);
\draw (-2.5,0)  .. controls  +(north:18 mm) and +(north:18 mm) ..   (1,0);
\draw (-1.5,0)  .. controls  +(north:5 mm) and +(north:5 mm) ..   (1,0);
    \end{tikzpicture}
 \caption{}
    \label{fig:max_rig_ic}
\end{figure}

For the second assertion, assume $\cM$ is MCIC. It suffices to show that any interval $J\subseteq\bI_{m,n}^*$ compatible with all intervals in $\cM$ belongs to $\cM$.
If $J$ is an integral interval, then $\cM\cup\{J\}$ is still interval consistent and compatible. Thus, the maximality of $\cM$ implies that $J\in\cM.$ If $J$ is non-integral, it follows immediately from the first assertion that $J\in\cM$.
\end{proof}

Applying the correspondence $\mathcal S$ between maximal rigid subcategories and maximal compatible sets mentioned in Proposition \ref{prop:supp}, we obtain the following results.
\begin{cor}\label{cor:MRIC_MCIC}
Each MRIC subcategory is a maximal rigid subcategory of $\rep(\bI_{m,n}^*)$.
Furthermore, the map ${\mathcal S}:  \MRIC(\bI_{m,n}^*)  \to \MCIC(\bI_{m,n}^*) $ is a bijection.
\end{cor}
%

Next two lemmas show that the half-integral intervals in a MCIC set imply the existence of some integral intervals intrinsically.
 
 \begin{lem}\label{lem:bound_int}
 Let $\cM$ be a MCIC set on $\bI_{m,n}^*$. Assume there exists an interval $|\alpha, \beta| \in\cM$, where $i^+\prec \alpha \prec (i+1)^-$, for some $m\leq i<n$. 
 \begin{enumerate}
 \item  If $\beta\succeq i+1$,  then $|i+1, \beta| \in\cM$;  
 \item  Either $|i^+,\beta|\in\cM$, or there exists $|\gamma,i|\in\cM$. 
 \end{enumerate}   
 \end{lem}
 \begin{proof}
 Since $\cM$ is interval consistent,  $\beta=j$ or $j^-$ for some integer $j$ and  $|\alpha', \beta| \in\cM$ for $i^+\prec\alpha'\prec (i+1)^-$. 
 
 (1) By Lemma \ref{lem:key} (3), it suffices to show that $|i+1, \beta|$ is compatible with all intervals in $\cM_{|\alpha, \beta|}:=\{J\in\cM: J\subseteq |\alpha, \beta|\}$.  In fact, because $\cM$ is interval consistent, we have $$\cM_{|\alpha, \beta|}=\cM_{|i+1, \beta|}\cup \{|\alpha', \beta| :  \alpha\preceq\alpha'\prec (i+1)^- \}.$$ 
 On one hand, for each $I\in\cM_{|i+1,\beta|}$, $I$ is compatible with $|i+1, \beta|$ since $I\subseteq |i+1, \beta|$. On the other hand,
 for each $\alpha\preceq\alpha'\prec (i+1)^- \}$, $|i+1, \beta|\subseteq |\alpha', \beta|$. Hence, we conclude that $|i+1, \beta|$ is compatible with all intervals in $\cM_{|\alpha, \beta|}$.
  
 %
%
%
 
 (2) Suppose $\cM$ does not contain any interval of the form $|\gamma,i|$. We claim that $|i^+,\beta|$ is compatible with all intervals in $\cM$.
 In fact, if there were  an interval $J=|\delta,\varepsilon|\in\cM$ incompatible with $|i^+,\beta|$, then either $\delta\prec i^+\preceq \suc(\varepsilon)\preceq \beta$ or $i^+\prec \delta\preceq \suc(\beta)\preceq \varepsilon$. 
 
 In the former case, since we assume $\cM$ does not contain intervals $|\gamma,i|$, it follows that $\varepsilon\neq i$, and hence $\suc(\varepsilon)\succ i^+$. Therefore, there exists some  $i^+\prec \alpha' \prec (i+1)^-$, such that $\delta\prec \alpha'\prec \suc(\varepsilon)\preceq \beta$. Thus, $J$ is incompatible with some $|\alpha', \beta|\in\cM$, a contradiction. In the later case, $i^+\prec \delta$ implies that there is some $i^+\prec \alpha' \prec (i+1)^-$, such that $\alpha'\prec \delta \preceq \suc(\beta)\preceq \varepsilon$. Hence, $J$ is again incompatible with some $|\alpha', \beta|\in\cM$, a contradiction.
 
 Therefore $|i^+,\beta|$ is compatible with all intervals in $\cM$. Due to the maximality of $\cM$, it contains $|i^+,\beta|$.
%
%
 \end{proof}

  \begin{lem}\label{lem:inf_int}
 Let $\cM$ be a MCIC set on $\bI_{m,n}^*$. If there is at least one interval $|\gamma,\beta|\in\cM$ with $\|\beta\|\in\bZ$, then $|\varphi(\beta),\beta|$ is an integral interval in $\cM$, where $\varphi(\beta)=\inf\{\gamma\in\bR^\pm: |\gamma,\beta|\in\cM\}$.
 \end{lem}
\begin{proof}
We prove by induction on $\|\beta\|$. Assume $\|\beta\|$ is the smallest integer in the interval $\bI=\bI_{m,n}^*$. In case $\bI=\bI_{m,n}$ or $\bI_{m,n}^-$, we have $\beta=m$ and there is only one possible interval  $|\gamma,\beta|=|m,m|\in\cM$ with right endpoint $\|\beta\|=m$. So $\varphi(\beta)=m$ and $|m,m|\in\cM$. In case $\bI=\bI_{m,n}^+$ or $\bI_{m,n}^\pm$, $\|\beta\|=m+1$. According to the interval consistency of $\cM$, the existence of the interval $|\gamma,\beta|$ implies intervals $|x,\beta|$ and $|x^+,\beta|\in\cM$ for all $x\in(m,m+1)$. Thus, $\varphi(\beta)=m^+$. It follows from  Lemma \ref{lem:bound_int} (2) that $|m^+,\beta|\in\cM$. The statement holds for both cases.

Suppose the statement holds for all $m\leq\|\beta\|\leq i$. Now consider $\|\beta\|=i+1$. If $|\varphi(\beta), \beta|\notin\cM$, then there is a limit sequence $\alpha\to \varphi(\beta)$ such that $|\alpha,\beta|\in\cM$. From the interval consistency  of $\cM$, it forces $\varphi(\beta)=l^+$ for some integer $l$ and for each $x\in (l,l+1)$, intervals $|x,\beta|$ and $|x^+,\beta|\in\cM$. According to Lemma \ref{lem:bound_int}  (2), since $|l^+,\beta|\notin\cM$, there exists an interval $|\gamma, l|\in\cM$.
Since $l\leq i$, by the inductive hypothesis, we know the interval $|\varphi(l),l|\in\cM$. We claim that $|\varphi(l),\beta|\in\cM$, and hence $\varphi(l)\preceq l\prec l^+=\varphi(\beta)$ contradicts with $\varphi(\beta)$ being the infimum.

\vskip5pt
Proof of the claim $|\varphi(l),\beta|\in\cM$:
\vskip5pt

Due to the maximality of $\cM$, it suffices to show that $|\varphi(l),\beta|$  is  compatible with any interval in $\cM$. Suppose otherwise, there is $J=|\delta,\varepsilon|\in\cM$ incompatible with $|\varphi(l),\beta|$. Then either $\delta\prec \varphi(l)\preceq \suc(\varepsilon)\preceq \beta$ or $\varphi(l)\prec \delta\preceq \suc(\beta)\preceq \varepsilon$.  We will conclude with contradictions.

In the former case, (i) when $\suc(\varepsilon)\preceq l$, we have $\delta\prec \varphi(l)\preceq \suc(\varepsilon)\preceq l$. So $J=|\delta,\varepsilon|\in\cM$ is incompatible with $|\varphi(l),l|\in\cM$, contradicts with the compatibility of $\cM$.
(ii) When $\suc(\varepsilon)=l^+$, i.e. $\varepsilon=l$. The interval $J=|\delta,\varepsilon|\in\{\gamma:|\gamma,l|\in\cM\}$. But then $\delta\prec\varphi(l)$ contradicts with the fact that $\varphi(l)$ being the infimum of this set.
(iii) When $\suc(\varepsilon)\succ l^+$, $\suc(\varepsilon)\succ x$ for some $x\in(l,l+1)$.  We have $\delta\prec x\preceq\suc(\varepsilon)\preceq \beta$. So $|\delta,\varepsilon|$ is incompatible with $|x,\beta|\in\cM$, contradicts with the compatibility of $\cM$ again.

In the latter case, (i) when $\delta\preceq l^+$, it follows that $\varphi(l)\prec \delta\preceq l^+=\varphi(\beta)\prec\suc(\beta)\preceq\varepsilon$. Hence $J=|\delta,\varepsilon|$ is incompatible with $|\varphi(l),l|\in\cM$.
(ii) When $\delta\succ l^+$, it follows that $\delta\succ x$ for some $x\in(l,l+1)$. Thus, $x\prec\delta\preceq \suc(\beta)\preceq \varepsilon$ and hence $J=|\delta,\varepsilon|$ is incompatible with $|x,\beta|\in\cM$. Both cases arrive at contradictions, which finishes the proof.
\end{proof}

We list the dual statements of the these lemmas and leave the proof for the reader.
 \begin{lem}\label{lem:bound_int_dual}
 Let $\cM$ be an MCIC set on $\bI_{m,n}^*$. Assume there exists an interval $|\alpha, \beta| \in\cM$, where $i^+\prec \beta \prec (i+1)^-$, for some $m\leq i<n$. 
 \begin{enumerate}
 \item  If $\alpha\preceq i$,  then $|\alpha,i| \in\cM$.  
 \item  Either $|\alpha, (i+1)^-|\in\cM$, or there exists $|i+1,\gamma|\in\cM$. 
 \end{enumerate}   
 \end{lem}

  \begin{lem}\label{lem:inf_int_dual}
 Let $\cM$ be a MCIC set on $\bI_{m,n}^*$. If there is at least one interval $|\alpha,\gamma|\in\cM$ with  $\|\alpha\|\in\bZ$, then $|\alpha,\psi(\alpha)|$ is an integral interval in $\cM$, where $\psi(\alpha)=\sup\{\gamma\in\bR^\pm: |\alpha,\gamma|\in\cM\}$.
 \end{lem}

The following observation reveals a crucial property of MCIC sets and its link to maximal compatible sets of integral intervals. It will be used later to build a correspondence between MRIC subcategories and tilting modules.  
 \begin{lem}\label{lem:m_int}
Let $\cM$  be an interval consistent set on $\bI_{m,n}^*$. Then $\cM$ is MCIC if and only if  $\widehat\cM:=\cM\cap\widehat\Intv(\bI_{m,n}^*)$ is a maximal compatible set of integral intervals.
 \end{lem}
 \begin{proof}
``If'' part: Assume $J$ is an interval compatible with all intervals in $\cM$. If $J$ is non-integral, then $J\in\cM$ by Lemma \ref{lem:max_rig_ic}. If $J$ is integral, then $J$ is compatible with $\widehat\cM$, which is maximal. Thus, $J\in\widehat\cM\subseteq\cM$ again.  Hence $\cM$ is maximal compatible.
 
``Only if'' part: Assume $I=|\mu,\nu|$ is an integral interval compatible with all intervals in $\widehat\cM$. We need to show that $I\in\widehat\cM$. Suppose otherwise $I\notin\widehat\cM$. Then $I$ must be incompatible with some half-integral intervals in $\cM$. Without loss of generality, assume $I$ is incompatible with $J=|\alpha,\beta|\in\cM$ for $i^+\prec\alpha\prec (i+1)^-$. Then either $\mu<\alpha\preceq\suc(\nu)\preceq \beta$ or $\alpha\prec \mu\preceq\suc(\beta)\preceq \nu$. To obtain a contradiction, we need to show that $I$ is incompatible with some interval in $\widehat\cM$.

 In the former case, because $\|\nu\|\in\bZ$, $\mu<\alpha\preceq\suc(\nu)\preceq \beta$ implies that $\mu\prec\alpha\prec (i+1)^-\preceq\suc(\nu)\preceq \beta$. Because $\suc(\nu)\neq(i+1)^-$, it follows that $\mu\prec\alpha\prec i+1\preceq\suc(\nu)\preceq \beta$. In particular, $\beta\succeq i+1$. It follows from Lemma \ref{lem:bound_int} (1) that $|i+1,\beta|\in\widehat\cM$ and the inequality shows that it is incompatible with $I$, which is a contradiction.

 In the latter case, $\alpha\prec \mu\preceq\suc(\beta)\preceq \nu$ implies that $\varphi(\beta)\prec\alpha\prec \mu\preceq\suc(\beta)\preceq \nu$, where $\varphi(\beta)=\inf\{\gamma: |\gamma,\beta|\in\cM\}$. According to Lemma \ref{lem:inf_int}, $|\varphi(\beta),\beta|\in\widehat\cM$ and the inequality implies that it  is  incompatible with $I$, which is again a contradiction.

One can prove for the case $i^+\prec\beta\prec (i+1)^-$ similarly, by Lemma \ref{lem:bound_int_dual} and Lemma \ref{lem:inf_int_dual}.
 \end{proof}

\begin{rmk}\label{rmk:nonbij}
 We warn the reader that if $\cM$ is not interval consistent on $(\bI,\bm\lambda)$,  $\widehat\cM$ may not be maximal compatible! For example, considering $\bI=(0,1)$ with critical vector $\bm\lambda=(0^+,\frac{1}{3},\frac{2}{3},1^-)$ and $\cM$ be the maximal compatible set in Example \ref{exm:1}, $ \widehat\cM=\{(0,1), [\frac{1}{3}, \frac{1}{3}], [\frac{2}{3}, \frac{2}{3}]\}$, which is compatible but not maximal compatible in $\widehat\Intv(\bI,\bm\lambda)$.
\end{rmk}

Next, we will discuss how to complete a maximal compatible set of integral intervals into a MCIC set. The key ingredients for such completions are those smallest intervals, which we call ``roofs''.

\begin{lem-def}\label{lemdef:exist1}
Let $\cN$ be a maximal compatible set of integral intervals on $\bI_{m,n}^*$. 
For each open interval $(i,i+1)$, $i=m,m+1,\cdots,n-1$, there exists a smallest (with respect to containment) interval  $|\sigma_i,\tau_i|\in\cN$ which contains $(i,i+1)$. Furthermore, the intervals $|\sigma_i,\tau_i|$, $i=m,m+1,\cdots, n-1$, are mutually distinct. We call $|\sigma_i,\tau_i|$ the {\bf roof} of $(i,i+1)$.
\end{lem-def}
\begin{proof}
Notice that $\cN$ is a finite set containing the largest interval $\bI_{m,n}^*$. Thus for each $(i,i+1)$ there exists at least one minimal element $I\in\cN$ which contains $(i,i+1)$. If $I=|\sigma_i,\tau_i|$ and  $I'=|\sigma'_i,\tau'_i|$ were two distinct minimal elements in $\cN$ containing $(i,i+1)$, then they don't contain each other. In order to be mutually compatible, they must be disjoint, which is a contradiction. Thus, there is a smallest interval $|\sigma_i,\tau_i|$ containing $(i,i+1)$.

Assume $|\sigma_i,\tau_i|=|\sigma_j,\tau_j|$ for some $m\leq i<j<n$. We claim that $J=|\sigma_i,j|$ would be an integral interval compatible with all intervals in $\cN$. But $J\notin \cN$ due to the minimality of $|\sigma_i,\tau_i|$. Hence it contradicts with the maximality of $\cN$.
\begin{figure}[h]
    \centering
    \begin{tikzpicture}
         \draw[-{To[scale=1.3]},line width=1pt, color=black] (-3.4, 0) -- (3.4, 0);
    \foreach \x in {-3,-2,-1,1,2,3}
        \draw[shift={(\x,0)},color=black, line width=1pt] (0pt,0pt) -- (0pt,5pt);
    \draw[shift={(-3,-0.4)},color=black] node[font=\large,text height=12pt] {$\sigma_i$};
    \draw[shift={(-2,-0.4)},color=black] node[font=\large,text height=12pt] {$i$};
     \draw[shift={(-1,-0.4)},color=black] node[font=\large,text height=12pt] {$i+1$};
    \draw[shift={(1,-0.4)},color=black] node[font=\large,text height=12pt] {$j$};
        \draw[shift={(2,-0.4)},color=black] node[font=\large,text height=12pt] {$j+1$};
    \draw[shift={(3,-0.4)},color=black] node[font=\large,text height=12pt] {$\tau_i$};

  \draw (-3,0)  .. controls  +(north:18 mm) and +(north:18 mm) ..(3,0);
\draw (-3,0)  .. controls  +(north:5 mm) and +(north:5 mm) ..   (-2,0);
\draw (2,0)  .. controls  +(north:5 mm) and +(north:5 mm) ..   (3,0);
\draw (-1,0)  .. controls  +(north:5 mm) and +(north:5 mm) ..   (1,0);
 \draw (-3,0)  .. controls  +(north:10 mm) and +(north:10 mm) .. node[above=1pt,text height=14pt,font=\small]{$J$} (1,0);

    \end{tikzpicture}
 \caption{}
    \label{fig:sigma_tau}
\end{figure}

To prove the claim, notice that by the minimality of $|\sigma_i,\tau_i|=|\sigma_j,\tau_j|$, for any interval $\cN\ni K\subseteq |\sigma_i,\tau_i|$, $K$ cannot contain the interval $(i,i+1)$ or $(j,j+1)$. So $K$ must be contained in either $|\sigma_i,i|$ or $|i+1, j|$ or $|j+1,\tau_i|$. Thus $J$ is compatible with all $\cN\ni K\subseteq |\sigma_i,\tau_i|$. Hence by Lemma \ref{lem:key} (3), $J$ is compatible with all intervals in $\cN$. This finishes the proof.
\end{proof}

\begin{lem}\label{lem:exist2}
Let $\cN$ be a maximal compatible set of integral intervals on $\bI_{m,n}^*$. Let  $|\sigma_i,\tau_i|\in\cN$ be the roof of $(i,i+1)$, for $i=m,m+1,\cdots,n-1$. For $m\leq i<j<n$, the following holds.
\begin{enumerate}
\item $\{|x,\tau_i|, |x^+,\tau_i|, |y,\tau_j|,|y^+,\tau_j|: x\in(i,i+1), y\in(j,j+1)\}$ is compatible;
\item $\{|x,\tau_i|, |x^+,\tau_i|, |\sigma_j,y|,|\sigma_j,y^-|: x\in(i,i+1), y\in(j,j+1)\}$ is compatible; 
\item $\{|\sigma_i,x|, |\sigma_i,x^-|, |y,\tau_j|,|y^+,\tau_j|: x\in(i,i+1), y\in(j,j+1)\}$ is compatible; 
\item $\{|\sigma_i,x|, |\sigma_i,x^-|, |\sigma_j,y|,|\sigma_j,y^-|: x\in(i,i+1), y\in(j,j+1)\}$ is compatible; 
\item  $\{|x,\tau_i|, |x^+,\tau_i|: x\in(i,i+1)\}\cup \cN$ is compatible;
\item  $\{|\sigma_i,x|, |\sigma_i,x^-|: x\in(i,i+1)\}\cup \cN$ is compatible.
\end{enumerate}
\end{lem}
\begin{proof}

(1) It is clear that both sets $\{|x,\tau_i|, |x^+,\tau_i|: x\in(i,i+1)\}$ and $\{|y,\tau_j|,|y^+,\tau_j|: y\in(j,j+1)\}$ are compatible.

If $|x,\tau_i|$ is incompatible with $|y,\tau_j|$, then $x\prec y\preceq \suc(\tau_i)\preceq \tau_j$. On one hand, $y\preceq \suc(\tau_i)$ implies that $\tau_i\succeq (j+1)^-$. On the other hand, $(i,i+1)\subseteq |\sigma_i,\tau_i|$ implies that $\sigma_i\prec x\prec j$. Thus, $\sigma_i\prec j\prec (j+1)^-\preceq \tau_i$, which means $(j,j+1)\subseteq|\sigma_i,\tau_i|$. If $\sigma_j\preceq\sigma_i$, then $|\sigma_i,\tau_i|\subsetneq |\sigma_j,\tau_j|$ is a smaller interval containing $(j,j+1)$, which is a contradiction. Hence, $\sigma_i\prec \sigma_j \preceq \suc(\tau_i)\preceq\tau_j$, contradicts with the compatibility of $|\sigma_i,\tau_i|$ and $|\sigma_j,\tau_j|$. Thus, $|x,\tau_i|$  and $|y,\tau_j|$ are compatible.

The compatibility of the rest pairs follows immediately.

(2) It is clear that both sets $\{|x,\tau_i|, |x^+,\tau_i|: x\in(i,i+1)\}$ and $\{|y,\tau_j|,|y^+,\tau_j|: y\in(j,j+1)\}$ are compatible.

If $|x,\tau_i|$ is incompatible with $|\sigma_j, y|$, then either  $x\prec \sigma_j\preceq \suc(\tau_i)\preceq y$ or $\sigma_j\prec x\preceq \suc(y)\preceq \tau_i$. In the former case, it follows that $\sigma_i \prec \sigma_j\preceq \suc(\tau_i)\preceq   \tau_j$, which implies that $|\sigma_i,\tau_i|$ and $|\sigma_j,\tau_j|$ are incompatible, a contradiction. In the latter case, it implies that both $|\sigma_i,\tau_i|$ and $|\sigma_j,\tau_j|$ contains intervals $(i,i+1)$ and $(j,j+1)$. Due to the minimality of roofs, it follows that $|\sigma_i,\tau_i|=|\sigma_j,\tau_j|$ is a roof of both $(i,i+1)$ and $(j,j+1)$. However, this contradicts with the fact that roofs are mutually distinct (Lemma-Definition \ref{lemdef:exist1}).

The compatibility of the rest pairs follows immediately.

Statement (3) is obvious and (4) can be proved similarly as (1).

(5) Notice that any integral interval $\cN\ni K\subseteq |\sigma_i,\tau_i|\}$ is contained in either $|\sigma_i,i|$ or $|i+1,\tau_i|$. Hence $|x,\tau_i|$ and $|x^+,\tau_i|$ are compatible with $K$. Therefore, by Lemma \ref{lem:key} (3), it implies that $|x,\tau_i|$ and $|x^+,\tau_i|$ are compatible with all intervals in $\cN$.

Statement (6) can be proved similarly as (5).
\end{proof}

\begin{rmk}\label{rmk:exist}
This lemma shows that starting from a maximal compatible set of integral intervals $\cN$, one can add half-integral intervals $\{|x,\tau_i|, |x^+,\tau_i|: x\in(i,i+1)\}$ or $\{|\sigma_i,x|, |\sigma_i,x^-|:x\in(i,i+1)\}$ to $\cN$ for each $m\leq i<n$, where $|\sigma_i,\tau_i|$ is the roof of $(i,i+1)$, such that the result $\cM$ is a MCIC set. Thus, there are at least $2^{n-m}$ ways one can complete $\cN$ into a MCIC subset and exactly one way one can complete it into a maximal rigid consistent above/below subset. We will show that these are all possible completions.
\end{rmk}

Recall from Lemma \ref{lem:m_int} that for an MCIC set $\cM$, $\widehat\cM$ is a maximal rigid set of integral intervals. So we can continue to use the notion of roofs for MCIC sets and call an interval $|\sigma_i,\tau_i|\in\cM$ a roof, if it is a roof in $\widehat\cM$.

\begin{lem}\label{lem:at_most_2}
Let $\cM$ be a MCIC set on $\bI_{m,n}^*$ containing $\cN$. Let $|\sigma_i,\tau_i|$ be the roof of $(i,i+1)$, for $i=m,m+1,\cdots,n-1$.
\begin{enumerate}
\item If there is an interval $|x,\tau|\in\cM$ for some $x\in (i,i+1)$,  then $\tau=\tau_i$.
\item If there is an interval $|\sigma,x|\in\cM$ for some $x\in (i,i+1)$,  then $\sigma=\sigma_i$.
\end{enumerate} 
\end{lem}
  \begin{proof}
We just prove (1). Statement (2) is similar.

By the compatibility of $\cM$, we have $\tau\preceq\tau_i$. By Lemma \ref{lem:inf_int}, there is an interval $|\varphi(\tau), \tau|\in\widehat\cM$, where $\varphi(\tau)=\inf\{\gamma\in\mathbb R^\pm: |\gamma,\tau|\in\cM\}$, which is   compatible with the roof $|\sigma_i,\tau_i|$. If $\tau\prec\tau_i$, it forces that $|\varphi(\tau),\tau|\subsetneq |\sigma_i,\tau_i|$, which contradicts with the  minimality of the roof.
  \end{proof}
  
 Combining results from Lemma \ref{lem:m_int}, Lemma \ref{lem:exist2} and Lemma \ref{lem:at_most_2}, we summarize the characterization of MCIC sets on $\bI_{m,n}^*$.
  
 \begin{thm}\label{thm:MCIC}
A set $\cM$ is MCIC on $\bI_{m,n}^*$ if and only if $\widehat\cM\in\widehat\MCom (\bI_{m,n}^*)$ and for each $m\leq i<n$ and $x\in(i,i+1)$ either $\{|x,\tau_i|, |x^+,\tau_i|: x\in(i,i+1)\}\in\cM$ or $\{|\sigma_i,x|, |\sigma_i,x^-|:x\in(i,i+1)\}\in\cM$, where $|\sigma_i,\tau_i|$ is the roof of $(i,i+1)$.  

In particular, the MCIC set $\cM$ is MCCB if and only if $\cM$ contains intervals $\{|\sigma_i,x|, |\sigma_i,x^-|:x\in(i,i+1)\}$ for each $m\leq i<n$ and $x\in(i,i+1)$. The MCIC set $\cM$ is MCCA if and only if $\cM$ contains intervals $\{|x,\tau_i|, |x^+,\tau_i|: x\in(i,i+1)\}\in\cM$ for each $m\leq i<n$ and $x\in(i,i+1)$.
 \end{thm}

Applying the bijection between MRIC subcategories and MCIC sets (Corollary \ref{cor:MRIC_MCIC}), we have the following result about completing maximal rigid subcategories of integral representations into MRIC subcategories.
\begin{cor}\label{cor:rigid_complete}
A subcategory $\fT$ is MRIC on $\bI_{m,n}^*$ if and only if $\widehat\fT=\fT\cap\bI_{m,n}^*\Int\in\MRig(\bI_{m,n}^*\Int)$  and for each $m\leq i<n$ and $x\in(i,i+1)$ either $\{k_{|x,\tau_i|}, k_{|x^+,\tau_i|}: x\in(i,i+1)\}\in\fT$ or $\{k_{|\sigma_i,x|}, k_{|\sigma_i,x^-|}:x\in(i,i+1)\}\in\fT$, where $|\sigma_i,\tau_i|$ is the roof of $(i,i+1)$.  

Each maximal rigid subcategory of $\bI_{m,n}^*\Int$ can be completed into $2^{n-m}$ different MRIC subcategories of $\rep(\bI_{m,n}^*)$, among which there is only one MRCB subcategory and one MRCA subcategory.     
\end{cor}
 
Finally we link MRIC subcategories with tilting modules, obtaining a full characterization. 

\begin{lem}\label{lem:tilt}
 There is a bijection between the set of isomorphism classes of tilting $kQ_n^*$-modules and maximal rigid subcategories of $\bI_{m,n}^*\Int$.
 \end{lem}
 \begin{proof}
By Proposition \ref{prop:tilting_eqv}, $T$ is a tilting $kQ_{m,n}$-module if and only if $\add T$ is a maximal rigid subcategory of $kQ_{m,n}^*\modu$. Denote by $kQ_{m,n}^*\tilt$ the isomorphism classes of tilting $kQ_{m,n}^*$-modules. Together with the  equivalence functor $\mathbf D_{m,n}^*$ in Lemma \ref{lem:eqv_A}, we have a bijection:    
\begin{align*}
kQ_{m,n}^*\tilt  & \to \MRig(\bI_{m,n}^*\Int)\\
[T] \mapsto  & \add \fD_{m,n}^*(T).
\end{align*}
 \end{proof}

 It follows immediately that $|kQ_{m,n}^*\tilt|=|\MRig(\bI_{m,n}^*\Int)|$  is  the well-known Catalan number \cite{BK}.
We show a formula for the number of MRIC subcategories over four types of continuous quivers of type $A$, which generalizes \cite[Theorem 2.5]{LZ1}.

\begin{cor}
Let $C_k=\displaystyle\frac{1}{k+1}{2k\choose k}$ be the $k$-th Catalan number. Then\\
{\rm(1)} $|kQ_{m,n}^*\tilt|=|\MRig(\bI_{m,n}^*\Int)|=|\MRCB((\bI_{m,n}^*)|=|\MRCA((\bI_{m,n}^*)|=\begin{cases} C_{2n-2m+1},  & *=\emptyset\\ 
  C_{2n-2m},  & *=+\\
   C_{2n-2m},  & *=-\\
  C_{2n-2m-1},  & *=\pm,\end{cases}$\\
{\rm(2)} $|\MRIC(\bI_{m,n}^*)|=2^{m-n}|kQ_{m,n}^*\tilt|.$
\end{cor}

\section{(Weak) tilting subcategories}

Recall the definition of (weak) tilting subcategories from Definition \ref{def:tilt_sub} and the definition of torsion pairs below:
\begin{defn}\label{def:tors}  Let $\cA$ be an abelian category. A  pair $(\cT,\cF)$ of full subcategories 
 of $\cA$ is called a {\bf torsion pair}, if the torsion and torsion-free classes $\cT$ and $\cF$, respectively, satisfying
\begin{enumerate}
\item $\Hom (\cT,\cF) = 0$, and
\item  for any object $M\in\cA$ there is a short exact sequence $0\to M' \to  M \to M''\to 0$ with
$M'\in \cT$ and $M''\in\cF$.
\end{enumerate}
\end{defn}

If $\cA$ is an abelian category and
$\fT\subseteq\cA$ is a full additive subcategory, write
\[
 \Gen(\fT)=\{X\in\cA:\text{$X$ is a quotient of an
 object of $\fT$}\}
\]
and
\[
 \fT^{\perp_1}=\{X\in\cA:\Ext^{1}_{\cA}(T,X)=0
 \text{ for every }T\in\fT\}.
\]
  Similar to tilting modules of finitely generated modules,  tilting subcategories can generate torsion pairs.
 \begin{lem}\cite[Proposition 2.22]{BBOS}\label{lem:weak_tilt_tor}
 Let $\cA$ be an abelian category with enough projective objects. If $\fT\subseteq \cA$ is weak tilting, then $\Gen(\fT)=\fT^{\perp_1}$. Furthermore, if $\fT$ is tilting, then $(\Gen(\fT),\fT^\perp)$ is a torsion pair.
 \end{lem}

The next proposition shows that weak tilting is stronger than maximal rigidity in a hereditary category with enough projective objects. Notice that the hypothesis on having enough projective objects is essential for constructing the commutative diagram.
 
\begin{prop}\label{prop:weak_tilt_rigid}
    Let $\cA$ be a hereditary abelian category with enough projective objects. If $\fT$ is a weak tilting subcategory of $\cA$, then it is maximal rigid.
\end{prop}
\begin{proof}
    Assume $\add(\fT\cup\{X\})$ is rigid. It suffices to show that $X\in\fT$.
    Suppose $0\to P_1\stackrel{f}\to P_0\stackrel{g}\to X\to 0$ is a projective resolution. Since $\fT$ is weak tilting, there are exact sequences $\delta:0\to P_0\stackrel{u}\to T_0\stackrel{v}\to T_1\to 0$ and $\delta':0\to P_1\stackrel{u'}\to T'_0\stackrel{v'}\to T'_1\to 0$, with $T_0,T_1, T'_0, T'_1\in\fT$. Hence, we have a pushout diagram:
    $$
    \xymatrix{0\ar[r]&P_1\ar[r]^f\ar@{=}[d]&P_0\ar[r]^g\ar[d]_u\ar@{}[dr] |{(p.o.)}&X\ar[r]\ar[d]&0\\
              0\ar[r]&P_1\ar[r]^h&T_0\ar[r]\ar[d]_v&E\ar[r]\ar[d]&0\\
              &&T_1\ar@{=}[r]&T_1}
    $$
Since $\Ext^1(\fT,X)=0$, it follows that $E\cong X\oplus T_1$ satisfying $\Ext^1(E,\fT)=0$. If we had shown $E\in\fT$, we could conclude that $X\in\fT$ due to the fact that $\fT$ is additive. To see $E\in\fT$, applying $\Hom(-,T_0)$ to the exact sequence $\delta'$,  the vanishing of $\Ext^1(T'_1,T_0)$ implies that $\Hom(f',T_0):\Hom(T'_0,T_0) \to \Hom(P_1,T_0)$ is surjective. In particular, the map $h$ factors through $f'$. Hence, there is a commutative diagram: 
$$
    \xymatrix{0\ar[r]&P_1\ar[r]^{f'}\ar@{=}[d]&T'_0\ar[r]^{g'}\ar[d]&T_1'\ar[r]\ar[d]&0\\
              0\ar[r]&P_1\ar[r]^h&T_0\ar[r]&E\ar[r]&0.}
$$
Thus $0\to T'_0\to T_0\oplus T_1'\to E\to 0$ is exact and it splits since $\Ext^1(E,\fT)=0$. Hence $E\in\fT$ because $\fT$ is additive. This finishes the proof.
\end{proof}

The converse fails in general. The obstruction is that maximal rigidity does not by itself provide the exact sequence required in (W3), which is in fact determined by a  monic left approximation.

\begin{prop}\label{prop:wtilt_app}
Let $\cA$ be a hereditary abelian category with enough projective objects. A maximal rigid subcategory $\fT$ of $\cA$ is weak tilting if and only if each projective object has a monic left $\fT$-approximation.  
\end{prop}
\begin{proof}
If $\fT$ is weak tilting, then by condition (W3), there is an exact sequence $0\to P\stackrel{f}\to T_0\to T_1\to 0,$ for each projective object $P$. Rigidity of $\fT$ implies that $\Hom(f,\fT)$ is surjective. Thus $f$ is a monic left $\fT$-approximation.

Conversely, weak tilting conditions (W1) and (W2) follow directly for any maximal rigid subcategory $\fT$ of a hereditary abelian category. For each projective object $P$, assume $f:P\to T_0$ is a monic left $\fT$-approximation and consider the short exact sequence 
$$0\to P\stackrel{f}\to T_0\to \cok f\to 0.$$
 Applying $\Hom(-,\fT)$ yields an exact sequence: $$0\to \Hom(\cok f,\fT)\to \Hom(T_0,\fT)\stackrel{\Hom(f,\fT)}\to \Hom(P,\fT)\to \Ext^1(\cok f, \fT)\to \Ext^1(T_0,\fT)=0$$ Since $f$ is a left $\fT$-approximation, the map $\Hom(f,\fT)$ is surjective. So $\Ext^1(\cok f,\fT)=0$. It is easy to see $\Ext^1(\fT,\cok f)=0$ as well, by applying $\Hom(\fT,-)$ to the short exact sequence. Thus due to the maximality of $\fT$, $T_1=\cok f\in\fT$. Therefore (T3) holds and $\fT$ is weak tilting.
\end{proof}

\begin{exm}\label{exm:max_rig_not_tilt}
Let $\fT=\add(\fT_1\cup \fT_2)$ be an additive subcategory of $\rep(\bI_{0,1}^{\pm})$, where
  $$\fT_1=\{k_{(\frac{m}{2^n},\frac{m+1}{2^n})}: n\geq 0, 0\leq m\leq 2^n-1\}, \quad \fT_2=\{k_{[x,x]}: x\in \mathbb R, x\neq\frac{m}{2^n}, n\geq 0, 0<m<2^n-1\}.$$ 
   Then according to Example \ref{exm:1}, $\fT$ is maximal rigid. We claim that it is  not weak tilting, because the projective object $P=k_{[a,1)}$ does not have a left $\fT$-approximation, for $a\neq\frac{m}{2^n}$. 
   
   Indeed, assume there were a left $\fT$-approximation $f: P\to T$ , where $T=\bigoplus\limits_{i=0}^k T_i$ is a finite direct sum of non-isomorphic indecomposable objects $T_i\in\fT$. Without loss of generality, we can assume $\pi_if\neq 0$ for each canonical projection $\pi_i:T\to T_i$. Consequently, $a\in\supp T_i$ for each summand $T_i$. Therefore, there is at most one summand, say $T_0$, belongs to $\fT_2$, which is isomorphic to $k_{[a,a]}$. The remaining direct summands $T_1,\cdots,T_k$ belong to $\fT_1$. One can always find some integers $n\gg 0$ and $m$ such that $a\in(\frac{m}{2^{n}},\frac{m+1}{2^{n}})\subsetneq\supp T_i$, for each $1\leq i\leq k$. Observe that $\Hom(T_i,k_{(\frac{m}{2^{n}},\frac{m+1}{2^{n}})})=0$ for $0\leq i\leq k$. However, this means that the nonzero map $P_{[a,1)}\to k_{(\frac{m}{2^{n}},\frac{m+1}{2^{n}})}$ cannot factor through $f$. 
\end{exm}  

We now return to representations of continuous type $A$ quivers. Interval consistency supplies the missing approximation data by forcing a uniform family of intervals in each open interval $(i,i+1)$. With this extra condition imposed, we are able to show every MRIC subcategory is in fact weak tilting. A preparatory lemma is needed in advance.

\begin{lem}[Horseshoe]\label{lem:horseshoe}
Let $\fT$ be a rigid subcategory of a hereditary abelian category $\cA$. Let $0\to A\stackrel{u}\to B\stackrel{v}\to C\to 0$ be an exact sequence in $\cA$.  Assume there are exact sequences $0\to A\stackrel{f}\to T_0\stackrel{f'}\to T_1\to 0$ and $0\to C\stackrel{h}\to T'_0\stackrel{h'}\to T'_1\to 0$ with $T_0,T_1,T'_0,T'_1\in \fT$ then there is a commutative diagram with exact rows and columns:
$$
\begin{tikzcd}
&0\arrow[d]&0\arrow[d]&0\arrow[d]\\
0\arrow [r]&A\arrow[r,"u"]\arrow [d,"f"]&B\arrow [r,"v"]\arrow [d,"g"]&C\arrow[r]\arrow[d,"h"]&0\\
0\arrow [r]&T_0\arrow[r,"\tiny{\begin{bmatrix} 1\\0\end{bmatrix}}"]\arrow [d,"f'"]&T_0\oplus T'_0\arrow[r,"\tiny{[0,1]}"]\arrow [d,"g'"]&T'_0\arrow[r]\arrow[d,"h'"]&0\\
0\arrow[r]&T_1\arrow[r,"\tiny{\begin{bmatrix} 1\\0\end{bmatrix}}"]\arrow[d]&T_1\oplus T'_1\arrow[r,"\tiny{[0,1]}"]\arrow[d]&T'_1\arrow[r]\arrow[d]&0,\\
&0&0&0
\end{tikzcd}
$$
\end{lem}
\begin{proof}
Applying $\Hom(-,T_0)$ to the sequence  $0\to C\stackrel{h}\to T'_0\stackrel{h'}\to T'_1\to 0$, we know $0=\Ext^1(T'_0,T_0)\to \Ext^1(C,T_0)\to 0$ is exact. Hence $\Ext^1(C,T_0)=0$. Then applying $\Hom(-,T_0)$ to the sequence $0\to A\to B\to C\to 0$, from the exact sequence $\Hom(B,T_0)\to \Hom(A,T_0)\to \Ext^1(C,T_0)=0$, we know $f$ factors through $u$. Put $g_1:B\to T_0$ such that $f=g_1u$ and $g_2=hv$. It is easy to check that $g=\tiny{\begin{bmatrix} g_1\\g_2\end{bmatrix}}:B\to T_0\oplus T'_0$ makes above two squares in the diagram commutative. By Snake Lemma, there is an exact sequence $0\to T_1\to \cok g\to T'_1\to 0$ associated with the cokernels. The sequence splits since $\fT$ is rigid, which gives rise to the desired commutative diagram.
\end{proof}

\begin{thm}\label{thm:MRCB_tilt} The following statements hold.
\begin{enumerate}
\item If $\fT$ is a MRIC subcategory of $\rep(\bI_{m,n}^*)$, then it is a weak tilting subcategory. 
\item If $\fT$ is a MRCB subcategory of $\rep(\bI_{m,n}^*)$, then it is a tilting subcategory. Moreover, $\Gen(\fT)$ is a torsion class.
\end{enumerate}
\end{thm}
\begin{proof}
 (1) We just prove for the continuous quiver $\bI_{m,n}$. Arguments for $\bI_{m,n}^{(+,-,\pm)}$ are similar.

Assume   $P=k_{[x,n]}$ is an indecomposable object in $\rep(\bI_{m,n})$, where $x\in(i,i+1)$ and $|\sigma_i,\tau_i|$ is the roof of $(i,i+1)$. 

Case 1: $k_{|\sigma_i,x|}\in\fT$. By Lemma \ref{lem:inf_int_dual}, there is an integral module $K=k_{|\sigma_i,\psi(\sigma_i)|}\in\fT$, where $\psi(\sigma_i)=\sup\{\gamma\in\bR^\pm: k_{|\sigma_i,\gamma|}\in\fT\}$. 
Denote by $M=k_{|\suc(\psi(\sigma_i)),n|}$ (set $M=0$ if $\psi(\sigma_i)=n$) and $N=k_{|x, \psi(\sigma_i)|}$. Then there is an exact sequence:
$$
0\to M\to P\to N\to 0.
$$
Since $M$ is a projective object in $\bI_{m,n}\Int$ and $\widehat\fT=\fT\cap\bI_{m,n}\Int$ is maximal rigid, due to the correspondence between tilting modules and maximal rigid subcategories of $\bI_{m,n}\Int$ (Lemma \ref{lem:tilt}), there is an exact sequence $0\to M \to T_0\to T_1\to 0$ for some integral representations $T_0$, $T_1\in \fT$.  Moreover, $N$ satisfies the weak tilting condition (W3), since it admits an exact sequence:
$
0\to N\to K \to k_{|\sigma_i,x^-|}\to 0,
$
with $K,k_{|\sigma_i,x^-|}\in\fT$. We can conclude that $P$ also satisfies (W3) by  Lemma \ref{lem:horseshoe} and Proposition \ref{prop:wtilt_app}.

Case 2: $k_{|x,\tau_i|}\in\fT$. Denote by  $M'=k_{|\suc(\tau_i), n|}$. Then there is an exact sequecne:
$$
0\to M'\to P\to k_{|x,\tau_i|}\to 0.
$$
Since $M'$ is a projective object in $\bI_{m,n}\Int$, for the same reason, there is an exact sequence $0\to M' \to T'_0\to T'_1\to 0$ for some integral representations $T'_0$, $T'_1\in \fT$. Moreover,  $k_{|x,\tau_i|}\in\fT$. By Lemma \ref{lem:horseshoe} and Proposition \ref{prop:wtilt_app}, we conclude that the condition (W3) holds for $P$.

(2) For any interval module $k_{|\alpha,\beta|}$, we need to find a right $\fT$-approximation. First, notice that if there is a nonzero map $g:k_{|\gamma,\delta|}\to k_{|\alpha,\beta|}$ from some $k_{|\gamma,\delta|}\in\fT$, then $\alpha\preceq \gamma \preceq \beta\preceq \delta$. Furthermore, since $\fT$ is consistent below, $\|\gamma\|$ is an integer. Precisely,  $\gamma$ belongs to the set $\mathcal G=|\alpha,\beta| \cap \{m,m^+, (m+1), \cdots, (n-1)^+, n\}$. 

Suppose the finite set $\mathcal G$ contains $t$ elements $\kappa_1,\cdots, \kappa_t$. Now for each $\kappa_j\in \mathcal G$, consider the set $\mathcal R_j=\{\delta: k_{|\kappa_j,\delta|}\in\fT, \beta\preceq \delta\}$. 
If all sets $\cR_j=\emptyset$, then we know $\Hom(\fT, k_{|\alpha,\beta|})=0$. In this case, the zero map is a right $\fT$-approximation.

If $\cR_j\neq\emptyset$ for some $1\leq j\leq t$, set $\omega_j=\inf\cR_j$ and we claim that $\omega_j\in\cR_j$. Suppose otherwise, there is a sequence $\{\delta_k\}\in\cR_j$ approaching to $\omega_j$. Then by the consistency condition, $\delta_k\in (i,i+1)$ for some $m\leq i< n$ when $k\gg0$ and $\omega_j=i^+$. Meanwhile, $\beta\leq \delta_k$ implies $\beta\leq \omega_j=i^+$. But since $\beta$ is the upper bound of an admissible interval, $\beta\neq i^+$. Thus, it follows that $i\geq \beta\geq \kappa_j$.  However, Lemma \ref{lem:bound_int_dual} implies that $k_{|\kappa_j, i|}\in\fT$, which contradicts $\omega_j$ being the infimum. Therefore $\omega_j\in\cR_j$ and $k_{|\kappa_j,\omega_j|}\in\fT$.

Take $f_j:k_{|\kappa_j, \omega_j|}\to k_{|\alpha,\beta|}$ to be the morphism with an image $k_{|\kappa_j,\beta|}$. We claim that the map
$$
f=(f_j):\bigoplus\limits_{\substack{j=1\\ \cR_j\neq\emptyset}}^t k_{|\kappa_j, \omega_j|}\to k_{|\alpha,\beta|}
$$
is a right $\fT$-approximation. In fact, for any nonzero map $g:k_{|\gamma,\delta|}\to k_{|\alpha,\beta|}$, it follows that $\gamma=\kappa_j$ for some $j$ and $\omega_j\preceq\delta\in\cR_j$. Thus, there is an epimorphism $u: k_{|\gamma,\delta|}\to k_{|\kappa_j,\omega_j|}$ such that $g=f_ju$. Hence, $g$ factors through $f$.

So $\fT$ is a contravariantly finite and therefore a tilting subcategory. By Lemma \ref{lem:weak_tilt_tor}, $\Gen(\fT)$ is a torsion class.
\end{proof}

%

We remark that among all the interval consistent but not consistent below subcategories, some of them are indeed tilting subcategories and some of them are not.
\begin{exm}
{\rm (1)} Consider the MRCA subcategory $\fT$, of $\rep(\bI_{0,1}):$
$$
\fT=\add\{k_{[0,0]}, k_{[0,1]}, k_{[1,1]}, k_{[x,1]}, k_{(x.1]}:x\in(0,1)\}.
$$
The representation $k_{(0,1]}$ does not admit a right $\fT$-approximation. So $\fT$ is not tilting.

{\rm (2)}  Consider the MRCA subcategory $\fU$, of $\rep(\bI_{0,1}):$
$$
\fU=\add\{k_{(0,1]}, k_{[0,1]}, k_{[1,1]}, k_{[x,1]}, k_{(x,1]}:x\in(0,1)\}.
$$
It is easy to see that for any interval module $k_{|\alpha,\beta|}\in\rep(\bI_{0,1})$, the epimorphism $f:k_{|\alpha,1|}\to k_{|\alpha,\beta|}$ is a right $\fT$-approximation.
\end{exm}

Moreover, not every tilting subcategory of $\rep(\bI_{m,n}^*)$ is interval consistent. 
\begin{exm}
Consider the subcategory $\fT$, of $\rep([0,+\infty)):$
$$
\fT=\add\{ k_{[n,+\infty)}, k_{[n,n]}, k_{[n,x]}, k_{[n,x)}:n\in\mathbb N, x\in(n,n+1)\}.
$$
One can check $\fT$ is a tilting subcategory. 

In fact, it is easy to see that $\fT$ satisfies weak tilting conditions (W1) and (W2). For each projective interval module $k_{|\alpha,+\infty|}$ with $n\prec \alpha\prec n+1$ for some integer $n\geq0$,   there is an exact sequence $0\to k_{|\alpha,+\infty|}\to k_{|n,+\infty|}\to k_{|n, \pre(\alpha)]}\to 0$, where both $T_0=k_{|n,+\infty|}$ and $T_1=k_{|n, \pre(\alpha)]}\in\fT$. And projective interval modules $k_{|n,+\infty|}$ are already in $\fT$. Thus the weak tilting conditions (W3) holds.

To see $\fT$ is contravariantly finite, let $k_{|\alpha,\beta|}$ be an interval module in $\rep([0,+\infty))$. If $|\alpha,\beta|\cap\mathbb N=\emptyset$, then $\Hom(\fT, k_{|\alpha,\beta|})=0$ and the zero map is the right $\fT$-approximation. If $\beta=+\infty$ and $m-1\prec \alpha\preceq m$ for some integer $m\geq0$, then the embedding $f: k_{|m,+\infty|}\to k_{|\alpha,+\infty|}$ is the right $\fT$-approximation. Otherwise, assume $m$ and $n$ are respectively the smallest and largest integer in the interval $|\alpha,\beta|$. Denote by $f_1:k_{|n,\beta|}\to k_{|\alpha,\beta|}$ the embedding and $f_2:k_{|m,+\infty|}\to k_{|\alpha,\beta|}$ the map with an image $k_{|m,\beta|}$. Consider the map
$$
f=(f_1,f_2): k_{|n,\beta|}\oplus k_{|m,+\infty|}\to k_{|\alpha,\beta|}.
$$
It is easy to see that $\ker f=k_{|n,+\infty|}$, $\Ima f=k_{|m,\beta|}$ and $\cok f=k_{|\alpha, m^-|}$ (set $k_\emptyset=0$ in case $\alpha\in\mathbb N$). Applying $\Hom(\fT,-)$ and examining the exactness shows that $\Hom(\fT,f)$ is surjective. Therefore $f$ is a right $\fT$-approximation.

The subcategory $\fT$ is not interval consistent in the sense of Definition \ref{def:int_consist}, because it has infinitely many critical points. 
\end{exm}

As the last object in this section, we will discuss the relation between maximal rigidity and the condition $\fT^{\perp_1}=\Gen(\fT)$ of subcategories $\fT$ in $\rep(\bR)$. It is surprising that the implication below does not need interval consistency of $\fT$.

\begin{prop}\label{prop:max_rig_perp}
Let $\fT$ be a subcategory of $\rep(\bR)$. If $\fT$ is maximal rigid  then $\fT^{\perp_1}=\Gen(\fT)$.
\end{prop} 
\begin{proof}
It is clear that $\Gen(\fT)\subseteq\fT^{\perp_1}$. It suffices to show the other inclusion.

For any interval module $k_{|\alpha,\beta|}\in\fT^{\perp_1}$, consider the nonempty set
$$
S=\{ \gamma'\in\bR^\pm:  \gamma'\succeq\beta, k_{|\alpha,\gamma'|}\in\fT^{\perp_1}\}.
$$
We claim that $\gamma:=\sup S\in S$.

Suppose $\gamma\not\in S$. Then there is a limit sequence $\{\gamma_k\}$ in $S$ approaching to $\gamma$. Hence, $\gamma=x^-$ for some real number $x$. Meanwhile,   there exists some $k_{|\mu,\nu|}\in\fT$ such that 
$\Ext^1(k_{|\mu,\nu|},k_{|\alpha,\gamma|})\neq 0$.  By Lemma \ref{lem:key}, it follows that $\mu\prec\alpha\preceq\suc(\nu)\preceq \gamma$.  
But $\suc(\nu)\preceq \gamma=x^-$ implies $\nu\prec x^-$ and hence $\nu\prec \gamma_k$ for some $k\gg0$. Therefore, we have an inequality $\mu\prec\alpha\preceq\suc(\nu)\preceq \gamma_k$. So $\Ext^1(k_{|\mu,\nu|},k_{|\alpha,\gamma_k|})\neq 0$, which is a contradiction.
 
Next, we will prove $\Ext^1(k_{|\alpha,\gamma|},\fT)=0$. Suppose otherwise, there is $k_{|\delta,\epsilon|}\in\fT$ such that $\Ext^1(k_{|\alpha,\gamma|},k_{|\delta,\epsilon|}) \neq 0$. Then $\alpha\prec \delta\preceq \suc(\gamma)\preceq \epsilon$ and there is an exact sequence, where we set $k_\emptyset=0$ if $\delta=\suc(\gamma)$:
$$
0\to k_{|\delta,\epsilon|}\to k_{|\alpha,\epsilon|} \oplus k_{|\delta,\gamma|}\to k_{|\alpha,\gamma|}\to 0.  
$$
Since $\fT^{\perp_1}$ is closed under extensions and direct summands, we conclude that $k_{|\alpha,\epsilon|}\in\fT^{\perp_1}$. But this implies that  $\epsilon\in S$, which contradicts with $\gamma$ being the supremum.

In summary, we have shown that $\Ext^1(\fT,k_{|\alpha,\gamma|})=0=\Ext^1(k_{|\alpha,\gamma|},\fT)$. Since $\fT$ is maximal rigid, it implies that $k_{|\alpha,\gamma|}\in\fT$.

Finally, there is clearly an epimorphism $k_{|\alpha,\gamma|}\to k_{|\alpha,\beta|}$. i.e. $k_{|\alpha,\beta|}\in\Gen(\fT)$. Therefore  $\fT^{\perp_1}\subseteq\Gen(\fT)$.
\end{proof}
 
The converse of Proposition \ref{prop:max_rig_perp} is not true even if $\fT$ is interval consistent.
\begin{exm}\label{exm:max_rig_not_perp1}
Let $\fT=\add\{k_{(0,2)}, k_{[1,1]}, k_{(0,x)}, k_{(0,x]}: x\in(0,1)\cup(0,2)\}$ be a subcategory of $\rep(\bI_{0,2}^\pm)$. It is straightforward to check that $\fT^{\perp_1}=\Gen\fT=\add(\fT\cup\{k_{(0,1)}, k_{(0,1]}\})$. But $\fT$ is not maximal rigid.
\end{exm}

Proposition \ref{prop:max_rig_perp} is not true for subcategories of an arbitrary hereditary abelian category. 
\begin{exm}\label{exm:max_rig_not_perp2}
Let $\cA= k[x]{\rm\text{-}nil}$ consists of nilpotent representations of $k[x]\modu$. It is easy to check that it is a hereditary abelian category and the nonzero indecomposable objects are classified by $M(n)\cong k[x]/\langle x^n\rangle$ for $n>1$, which are not rigid. Thus, $\{0\}$ is a maximal rigid subcategory of $\cA$. However, $0^{\perp_1}=\cA\neq \Gen(0)=0$.
\end{exm}

It is an open question if Proposition \ref{prop:max_rig_perp} holds for hereditary abelian categories with enough projective objects.

\section{Support tilting and torsion classes}
\subsection{Integrally supported maximal rigid subcategories}
The support of a finitely generated left module $M$ over a finite dimensional algebra $\Lambda$, is the set of simple $\Lambda$-modules that occur in the composition series of $M$. Following \cite{ONFR}, the support algebra of $M$  is $\Lambda_M:=\Lambda/I_M$, where $I_M$ is the two-sided ideal generated by all idempotents $e$ such that $eM=0$. A left $\Lambda$-module $T$ is called {\bf support tilting}, if $T$ is tilting considered as a $\Lambda_T$-module.




Let $Q$ be a finite acyclic quiver and $T$ be a support tilting  $kQ$-module. Then the support algebra $(kQ)_T\cong kC$, where $C$ is the minimal subquiver of $Q$ on which $T$ is supported. Let $C_1,\cdots, C_s$ be connected components of $C$ and denote by $T|_{C_i}$ the largest summand of $T$ supported on $C_i$ which will be referred as {\bf support components}. With these notations we have the following result.

\begin{prop}\label{prop:supp_tilt}
	A partial tilting module $T$ is support tilting if and only if each support component $T|_{C_i}$ is a tilting $kC_i$-module.
\end{prop}
\begin{proof}
As a $kC\cong \prod\limits_{i=1}^s kC_i$-module, $T\cong \bigoplus\limits_{i=1}^s T|_{C_i}$. Hence, $T$ is a tilting module if and only if $T|_{C_i}$ is a tilting $kC_i$-module for each component.
\end{proof} 
In particular, when $T$ is a support tilting $kQ_{m,n}^*$-module, the support algebra of $T$ is isomorphic to a product $\prod\limits_{i=1}^s kQ_{m_i,n_i}^* $ where the subquivers $Q_{m_i,n_i}^*$'s are disjoint intervals on $Q_{m,n}^*$. Furthermore, they must be mutually compatible, since $\Ext^1(T|_{Q_{m_i,n_i}^*},T|_{Q_{m_j,n_j}^*})=0$ for all support components. This leads to the following analogous notion for continuous representations.

\vskip5pt
This componentwise description has a direct continuous analogue. The support of a subcategory $\fT$ of $\rep(\bR)$ is $\Supp \fT:=\{x\in\bR: T(x)\neq 0$ for some $T\in\fT\}$. A subcategory $\fT$ is called {\bf finitely supported}, if $\Supp \fT=\bJ_1\cup \bJ_2\cup \cdots \cup \bJ_s$ is a disjoint union of finitely many mutually compatible intervals. Let $\fT|_{\bJ_i}$ be the full subcategory of $\fT$ consisting of objects whose supports are contained in $\bJ_i$, which will be referred to as the support component of $\fT$.
A  subcategory $\fT$ is called finitely supported maximal rigid (or weak tilting, etc.) if $\fT$ is finitely supported and each support component $\fT|_{\bJ_i}$ is maximal rigid (or weak tilting,  etc.) on $\bJ_i$.

Especially for continuous quivers of type $A$, a subcategory $\fT$ is said to be {\bf integrally supported} if $\Supp \fT=\bI_{m_1,n_1}^*\cup \bI_{m_2,n_2}^*\cup \cdots \cup \bI_{m_s,n_s}^*$ is a disjoint union of  mutually compatible integral intervals. Here, for each component, $*$ could be independently one of the four interval types. With the help of Proposition \ref{prop:homeo_int}, one can always find a categorical equivalence, which sends a finitely supported subcategory of $\rep(\bR)$ to an integrally supported subcategory of $\rep(I_{m,n}^*)$. Similarly, a subcategory $\fT$ of $\rep(I_{m,n}^*)$ is said to be integrally supported maximal rigid (or weak tilting, ect.) if $\fT$ is integrally supported and each support component $\fT|_{\bI_{m_k,n_k}}$ is maximal rigid (or weak tilting,  ect.) of $\rep(\bI_{m_k,n_k}^*)$.
 
 
\begin{thm}\label{thm:supp_MRCB}
There is a bijection between the set of isomorphism classes of support tilting $kQ_{m,n}^*$-modules and the set of integrally supported MRCB subcategories of $\rep(\bI_{m,n}^*)$. Furthermore, this correspondence sends each support tilting module $T$ to an integrally supported MRCB subcategory $\fT$ such that $\widehat\fT=\add\fD_{m,n}^*(T)$.
\end{thm}
\begin{proof}
For a support tilting  $kQ_{m,n}^*$-module $T$, assume its support algebra is $\prod\limits_{i=1}^s kQ_{m_i,n_i}^*$. According to Lemma \ref{lem:tilt}, $\add\fD_{m_i,n_i}^*(T|_{Q_{m_i,n_i}^*})$ is a maximal rigid subcategory of $\bI_{m_i,n_i}^*\Int$, which can be completed uniquely into a MRCB subcategory $\fT_i$ of $\rep(\bI_{m_i,n_i}^*)$ due to Corollary \ref{cor:rigid_complete}. We define the map 
$$
\Phi: \{\begin{smallmatrix}\text{iso. classes of support}\\ \text{tilting\ } kQ_{m,n}^*\text{-modules} \end{smallmatrix} \} \longrightarrow   \{\begin{smallmatrix}\text{integrally supported MRCB}\\ \text{subcategories of } \rep(\bI_{m,n}^*) \end{smallmatrix}\} 
$$
by $\Phi([T])=\add (\bigcup\limits_{i=1}^s\fT_i)$. It is easy to see that $$\widehat\fT=\add(\bigcup\limits_{i=1}^s\widehat\fT_i)= \add(\bigoplus\limits_{i=1}^s\fD_{m_i,n_i}^*(T|_{Q_{m_i,n_i}^*}))=\add\fD_{m,n}^*(T).$$

Conversely, for an integrally supported MRCB subcategory $\fT$, assume $\Supp(\fT)=\bigcup\limits_{i=1}^s \bI_{m_i,n_i}^*$. By definition,  $\fT|_{\bI_{m_i,n_i}^*}$ is a MRCB subcategory of $\rep(\bI_{m_i,n_i}^*)$. Thus $\widehat{\fT_i}=\fT|_{\bI_{m_i,n_i}^*}\cap\bI_{m_i,n_i}^*\Int$ is a maximal rigid subcategory of $\bI_{m_i,n_i}^*\Int$.  According to Lemma \ref{lem:tilt}, there exists a tilting $kQ_{m_i,n_i}^*$-module $T_i$ such that $\add\fD_{m_i,n_i}^*(T_i)\cong\widehat{\fT_i}$.  We define the map 
$$
\Psi: \{\begin{smallmatrix}\text{integrally supported MRCB}\\ \text{subcategories of } \rep(\bI_{m,n}^*) \end{smallmatrix}\} 
\longrightarrow \{\begin{smallmatrix}\text{iso. classes of support}\\ \text{tilting\ } kQ_{m,n}^*\text{-modules} \end{smallmatrix} \} $$
by $\Psi(\fT)=[\bigoplus\limits_{i=1}^sT_i]$.

Because $\widehat{\fT_i}$ can only be completed into the MRCB subcategory $\fT|_{\bI_{m_i,n_i}^*}$, we conclude that $\Phi\circ\Psi$ is an identity. Since the maps are between finite sets, they are bijections.
\end{proof}

By Theorem \ref{thm:MRCB_tilt}, an integrally supported MRCB subcategory is always an integrally supported tilting subcategory. Next we  will study the corresponding torsion class.

\subsection{Integrally generated torsion classes}  
 A torsion class $\cT$ in an abelian category is {\bf finitely generated} if there exists and object $T$, which is a finite sum of indecomposable objects, such that $\cT=\Gen(T)$. We say an object $X$ in $\rep(\bI_{m,n}^*)$ is  {\bf integrally generated} if $X\in\Gen(T)$,  for some integral representation $T$. Similarly, a torsion class $\cT$ of  $\rep(\bI_{m,n}^*)$ is  integrally generated if $\cT=\Gen(\widehat T)$,  for some integral representation $\widehat T$. 
 
 Using Proposition \ref{prop:homeo_int}, one can show the following result, which indicates the ubiquity of integrally generated torsion classes.
 
 \begin{prop}
 Let $\bI$ be an interval on $\bR$ and $\fT$ be a finitely generated torsion class in $\rep(\bI)$. Then there exists an equivalence $F:\rep(\bI)\to\rep(\bI_{m,n}^*)$ induced by some order preserving homeomorphism  $f:\bI\to\bI_{m,n}^*$ such that  $F(\fT)$ is integrally generated.
 \end{prop}
 
 Recall from Theorem \ref{thm:MRCB_tilt} that for each  MRCB subcategory $\fT$, $\Gen(\fT)$ is a torsion class. We point out that it is in fact an integrally generated torsion class.  
\begin{lem}\label{lem:mrcb_is_int}
Let  $\fT$  be a MRCB subcategory of $\rep(\bI_{m,n}^*)$. Then $\widehat \fT$  has  an additive generator $\widehat T$ such that $\Gen(\fT)=\Gen(\widehat T)$.
\end{lem}
\begin{proof}
 Clearly $\widehat\cM:=\{I\subseteq \bI_{m,n}^*: k_I\in \widehat\fT\}\subseteq \widehat\Intv(\bI_{m,n}^*)$ is a finite set.  The subcategory $\widehat \fT$  has  an additive generator $\widehat T=\bigoplus\limits_{I\in\widehat\cM}k_I$.

Assume $k_{|\alpha,x|}\in\fT$ where $i<x<i+1$. By Theorem \ref{thm:MCIC}, there is a roof $(\sigma_i,\tau_i)$ of the open interval $(i,i+1)$ and $\alpha=\sigma_i$. Hence $k_{|\sigma_i,\tau_i|}\to k_{|\alpha,x|}$ is an epimorphism, where $k_{|\sigma_i,\tau_i|}\in\widehat\fT$ is isomorphic to a direct summand of $\widehat T$. Therefore $\fT$ is generated by $\widehat T$. Consequently,  $\Gen(\fT)=\Gen(\widehat T)$.
\end{proof}

 The classical Ingalls-Thomas bijection shows that the map $T\mapsto \Gen(T)$ is a bijection between support tilting modules and finitely generated torsion classes over path algebras. The theorem below shows that the same map gives rise to a bijection between integrally supported MRCB subcategories and integrally generated torsion classes, which can be considered as a continuous version of Ingalls-Thomas bijection.
  
%
\begin{thm}\label{thm:IT_contin}
The map $\fT\mapsto \Gen(\fT)$ yields a bijection between the set of integrally supported MRCB subcategory of $\rep(\bI_{m,n}^*)$ and the set of  integrally generated torsion class of $\rep(\bI_{m,n}^*)$.
\end{thm}

\begin{proof}
    Let $\fT$ be an integrally supported MRCB subcategory of $\rep(\bI_{m,n}^*)$. Assume $\Supp(\fT)=\bigcup\limits_{i=1}^s \bI_{m_i,n_i}^*$. By definition,  $\fT_i=\fT|_{\bI_{m_i,n_i}^*}$ is a MRCB subcategory of $\rep(\bI_{m_i,n_i}^*)$. On one hand, by Theorem \ref{thm:MRCB_tilt}, $\Gen(\fT_i)$ is a torsion class. Hence $\Gen(\fT)=\add(\bigcup\limits_{i=1}^s \Gen(\fT_i))$ is a torsion class. On the other hand, since $\fT_i$ is consistent below, due to Lemma \ref{lem:mrcb_is_int},  $\Gen(\fT_i)=\Gen(\widehat T_i)$ for some additive generator $\widehat{T_i}$ of $\widehat \fT_i$. Therefore, $\Gen(\fT)=\Gen(\bigoplus\limits_{i=1}^s\widehat{T_i})$ is  integrally generated.
    
 Conversely, let $\cT$ be an integrally generated  torsion class. Then $\widehat\cT=\cT\cap\bI_{m,n}^*\Int$ is a torsion class in $\bI_{m,n}^*\Int$ and $\cT=\Gen(\widehat\cT)$. Using the equivalence Lemma \ref{lem:eqv_A},  $\widehat\cT=\fD_{m,n}^*(\cT')$ for some torsion class $\cT'$ of $kQ_{m,n}^*\modu$. Theorem \ref{thm:IT} implies that $\cT'=\Gen(T)$ for some support tilting $kQ_{m,n}^*$-module $T$.   By Theorem \ref{thm:supp_MRCB}, it corresponds to a unique integrally supported MRCB subcategory $\fT$ satisfying $\widehat\fT=\add\fD_{m,n}^*(T)$. Finally, we need to check $\cT=\Gen(\fT)$.  In fact,
 $$\cT=\Gen(\widehat\cT)=\Gen(\fD_{m,n}^*(\cT'))=\Gen(\fD_{m,n}^*(T))=\Gen(\widehat\fT)=\Gen(\fT)$$ and therefore the map $\fT\mapsto \Gen(\fT)$ is a bijection.
\end{proof}

{\bf Acknowledgments:} We would like to thank  Xiu-Hua Luo and Jing Ma for helpful discussions and comments.

\end{document}